\documentclass{amsart}
\usepackage{amsmath,amssymb}
\usepackage{float}
\usepackage{color}
\usepackage{graphicx}

\usepackage{enumerate}
\numberwithin{equation}{section}
\usepackage{hyperref}

\newcommand{\var}{\operatorname{var}}
\newcommand{\Int}{\operatorname{int}}
\newcommand{\esup}{\operatorname{ess\,sup}}

\newcommand{\elim}{\operatorname{ess\,lim}}

\newcommand{\Reg}{\operatorname{Reg}}
\newcommand{\fr}{\operatorname{frac}}

\def\j{R}

\def\eq{equation}

\def\tk{\widetilde{K(r)}}

\def\ep{\varepsilon}

\def\vphi{\varphi}
\def\Si{\Sigma}
\def\s{\sigma}
\def\o{\omega}
\def\Sn{\Sigma_n}
\def\St{\Sigma_*}
\def\Sb{\Sigma_b}
\def\o{\omega}
\def\O{\Omega}
\def\r{\rho}
\def\co{\overline{O}}

\def\<{\langle}
\def\>{\rangle}

\def\E{\mathbb{E}}
\def\N{\mathbb{N}}
\def\P{\mathbb{P}}
\def\R{\mathbb{R}}
\def\rd{\mathbb{R}^d}
\def\CB{\mathcal{B}}

\def\K{\mathcal{K}}
\def\L{\mathcal{L}^d}
\def\F{\mathcal{F}}
\def\H{\mathcal{H}}
\def\C{\mathcal{C}}
\def\PR{\mathcal{PR}}

\def\1{\mathbf{1}}
\newcommand{\const}{\operatorname{const}}
\newcommand{\Sim}{\operatorname{Sim}}
\newcommand{\nor}{\operatorname{nor}}
\newcommand{\rea}{\operatorname{reach}}

\newcommand{\dist}{\operatorname{dist}}

\newcommand{\INt}{\operatorname{int}}
\newcommand{\conv}{\operatorname{conv}}
\newcommand{\reach}{\operatorname{reach}}
\newcommand{\Tan}{\operatorname{Tan}}
\newcommand{\Nor}{\operatorname{Nor}}

\def\tit{\textit}
\theoremstyle{plain}
   \newtheorem{thm}{Theorem}[section]
   \newtheorem{thms}{Theorem}[section] 

  \newtheorem{lems}[thms]{Lemma}
   
   \newtheorem{props}[thms]{Proposition}

\theoremstyle{remark}
\newtheorem*{remark*}{Remark}
\theoremstyle{definition}

   \newtheorem{rems}[thms]{Remark}
   
\newtheorem{ex}[thms]{Example}

\begin{document}

\title{Mean Lipschitz-Killing curvatures for homogeneous random fractals
}

\author{Jan Rataj}\curraddr{\it Jan Rataj, Charles University, Faculty of Mathematics and Physics, Soko\-lovs\-k\'a 83, 186 75 Praha 8, Czech Republic}
\author{Steffen Winter}\curraddr{\it Steffen Winter, Institute of Stochastics, Karlsruhe Institute of Technology, Englerstr. 2, D-76131 Karlsruhe, Germany}
\author{Martina Z\"ahle}\curraddr{\it Martina Z\"ahle, Institute of Mathematics,  University of Jena, Ernst-Abbe-Platz 2, D-07743 Jena, Germany.}

\begin{abstract}
Homogeneous random fractals form a probabilistic extension of self-similar sets with more dependencies than in random recursive constructions. For such random fractals we consider mean values of the Lipschitz-Killing curvatures of their parallel sets for small parallel radii. Under the Uniform Strong Open Set Condition and some further geometric assumptions we show that rescaled limits of these mean values exist as the parallel radius tends to $0$. Moreover, integral representations are derived for these limits which extend those known in the deterministic case.
\end{abstract}

\thanks{Steffen Winter was supported by DFG grant WI 3264/5-1 and Martina Z\"ahle by DFG grant ZA 242/8-1.}





\maketitle

\section{Introduction}\label{intro}

Fractal versions of the $k$-th order Lipschitz-Killing curvatures $C_k$ in $\rd$ known from convex geometry, differential geometry and geometric measure theory have been considered in \cite{Wi08,WZ13} and subsequent papers for deterministic self-similar sets and in \cite{Za11} for random recursive constructions. Compared to the latter class of models, homogeneous random fractals, as first considered in \cite{Ha92} (for a special case), possess more dependencies in their structure. This leads, in particular, to the phenomenon that their a.s.\ Minkowski dimension, which agrees again with the a.s.\ Hausdorff dimension, is determined by a different equation than in the recursive case (see \cite{RU11} together with \cite{Tr17} for the general case). Moreover, in \cite{Tr21a} it is shown that in the non-deterministic case there exists no gauge function for an exact Hausdorff measure. It was conjectured in \cite{Za20} that the almost sure Minkowski content does not exist either. Meanwhile this has been proved in \cite{Tr21b}. In contrast to this striking difference in the almost sure behaviour of both models, it turned out that for expectations the results are the same as in the case of stochastically self-similar sets, i.e., random recursive constructions (see \cite{Za20}). 

The Minkowski content may be viewed as the marginal case $k=d$ of the fractal Lipschitz-Killing curvatures.  It is the aim of the present paper to extend the above mentioned results regarding Minkowski contents of homogeneous random fractals $F$ to all Lipschitz-Killing curvatures.


For a deterministic self-similar set $F$ satisfying the open set condition and some additional assumptions, Lipschitz-Killing curvatures can be introduced by approximation with parallel sets. If $r>0$ is a regular value of the distance function $d(\cdot,F)$ of $F$, then the $r$-parallel set $F(r):=\{x:\, d(x,F)\leq r\}$ has `nice' geometric properties. (More precisely, its boundary is a Lipschitz manifold and the closure of its complement has positive reach.) In particular, $F(r)$ admits Lipschitz-Killing curvature measures $C_k(F(r),\cdot)$, $k=0,\dots,d-1$. We use the notation $C_k^{\var}(F(r),\cdot)$ for the variation measures, and $C_k(F(r)):=C_k(F(r),\R^d)$ for the total mass.
Provided that almost all $r>0$ are regular values (which is always the case if $d\leq 3$, see \cite{Fu85}),
the (total) \emph{fractal Lipschitz-Killing curvatures} of $F$ can be introduced as rescaled (essential or averaged) limits
$$
\elim_{\ep\rightarrow 0}\limits\ep^{D-k}\,C_k(F(\ep)) \qquad \text{ or }\qquad \lim_{\delta\rightarrow 0}\frac{1}{|\ln\delta|}\int_\delta^1\ep^{D-k}\, C_k(F(\ep))~\ep^{-1}d\ep,
$$
where $D$ is the Minkowski dimension of $F$.
It turned out that the first one (essential limit) exists in the case of non-lattice self-similar sets, while the second one (averaged limit) exists in general, see \cite{Wi08} and \cite{Za11} for details.

In order to give some idea of the results obtained in the paper, we give an informal description of the considered model, a detailed definition can be found in Section~\ref{sec:def_main_res}. Consider a random iterated function system (random IFS) $f=(f_1,\dots,f_N)$ consisting of a random number $N\geq 2$ of contracting similarities $f_1,\dots,f_N$ with (random) contraction ratios $r_1,\dots,r_N$ fulfilling $0<r_{\min}\leq r_i\leq r_{\max}<1$ for some deterministic values $r_{\min},r_{\max}$. We assume the \emph{Uniform Open Set Condition} (UOSC): there exists a (deterministic) nonempty open set $O$ such that almost surely, $f_i(O)\subset O$ and $f_i(O)\cap f_j(O)=\emptyset$, $i\neq j\leq N$ (cf.\ \eqref{UOSC}). The \emph{homogeneous random fractal} $F$ is defined by means of an i.i.d.\ random sequence $(f^n)_{n\in\N}$, where each random IFS $f^n=(f^n_1,\dots,f^n_{N^n})$ has the same distribution as $f$. Then $F$ is obtained by applying at each construction step $n\in\N$ the same chosen IFS $f^n$ to all components of that step. More precisely, denoting $\Sigma_n:=\{\sigma_1\sigma_2\dots\sigma_n:\, 1\leq\sigma_i\leq N^i, 1\leq i\leq n\}$, $n\in\N$, and $f_\sigma:=f^1_{\sigma_1}\circ\dots\circ f^n_{\sigma_n}$, $\sigma\in\Sigma_n$, the homogeneous random fractal associated to the random IFS $f$ is the random set $F$ defined by
$$
F:=\bigcap_{n=1}^\infty\bigcup_{\sigma\in\Sigma_n}f_\sigma(\overline{O}).
$$
If $\E N<\infty$ (which we assume here throughout), then the equation $\E\sum_{i=1}^N r_i^D=1$ has a unique solution $D\in[0,d]$, called the \emph{mean Minkowski dimension} of $F$. (Note that it is given by the same formula as in the random recursive case, cf.~\cite{Za20} and \eqref{minkdim} below.) We consider the measure $\mu:=\E\sum_{i=1}^N \1\{|\ln r_i|\in\cdot\} r_i^D$ and its mean value $\eta:=\E\sum_{i=1}^N |\ln r_i|\,r_i^D.$
Further, we let $\Sigma_*:=\bigcup_{n=1}^\infty\Sigma_n$, 
and for any $r>0$, we denote by $\Sigma(r)$ the set of all finite words $\sigma=\sigma_1\dots\sigma_n\in\Sigma_*$ such that $r^1_{\sigma_1}\cdots r^n_{\sigma_n}\leq \frac{r}{2|O|}< r^1_{\sigma_1}\cdots r^{n-1}_{\sigma_{n-1}}$. Here $|O|$ denotes the diameter of $O$. For $\sigma\in\Sigma_*$ let $O_\sigma:=f_\sigma(O)$. Observe that $O_\sigma$ is open, while its parallel sets $O_\sigma(r)$, $r\geq 0$, are closed.

Our main result Theorem~\ref{maintheorem} establishes the existence of fractal Lipschitz-Killing curvatures for a large class of homogeneous random fractals. The following statement is a special case, as can easily be seen from Proposition~\ref{prop:easier-cond}\,(d). It provides this existence under a stronger but easier to state hypothesis.

\begin{thm} \label{thm:1.1}
Let $F\subset\R^d$ be a homogeneous random fractal with $\E N<\infty$ satisfying UOSC with some open set $O\subset\R^d$ such that $F\cap O\neq\emptyset$ almost surely. Let $k\in\{0,1,\ldots,d\}$ and assume further that
\begin{enumerate}
\item[(i)] with probability one, almost all $r>0$ are regular values of $d(\cdot,F)$;
\item[(ii)] there exists a constant $c>0$ such that with probability one,
$$C_k^{\var}\left(F(r),O_\sigma(r)\cap O_{\sigma'}(r)\right)\leq cr^k, \text{ for a.a. } r>0 \text{ and all } \sigma,\sigma'\in\Sigma(r) \text{ with } \s\neq\s'.$$ 
\end{enumerate}
Define $
F_i:=\bigcap_{n=1}^\infty\bigcup_{\sigma\in\Sigma_n, \sigma_1=i} f_\sigma(\overline{O})$ for $i\in\{1,\ldots,N\}$ and  
$$R_{k,1}(r):=\E C_k(F(r))-\E\sum_{i=1}^N\1_{(0,r_i]}(r)\, C_k\big(F_i(r)\big), \quad r>0.$$
Then,
$$\overline{C}_{k,F}^{\fr}:=\lim_{\delta\rightarrow 0}\frac{1}{|\ln\delta|}\int_\delta^1\ep^{D-k}\E\, C_k(F(\ep))~\ep^{-1}d\ep = \frac{1}{\eta}\int_0^{1} r^{D-k-1}R_{k,1}(r)\, dr.$$
Moreover, if the measure $\mu$ is non-lattice, then
$$C_{k,F}^{\fr}:=\elim_{\ep\rightarrow 0}\limits\ep^{D-k}\E\,C_k(F(\ep))
=\overline{C}_{k,F}^{\fr} .
$$
\end{thm}

As in the deterministic case, the values $\overline{C}_{k,F}^{\fr}$ (and $C_{k,F}^{\fr}$, in case they exist) can be interpreted as the \emph{mean (averaged) fractal Lipschitz-Killing curvatures} of $F$ of order $k$. (Note that in contrast to common differential geometric terminology in our paper the term `mean' curvatures is used in the probabilistic sense of expectations.) For the cases $k=d-1$ and $k=d$, conditions (i) and (ii) are in fact not needed.

Our result covers the case of deterministic self-similar sets considered in \cite{Wi08} and \cite{Za11}. As a nontrivial example, we discuss below a family of random homogeneous fractals, for which the associated random IFS choose between two deterministic IFS generating Sierpi\'nski type gaskets, see Section~\ref{sec:ex}. In this case, the values of $C_{k,F}^{\fr}$ are determined explicitly.

The plan of the paper is as follows.
In Section~\ref{sec:curv parallel sets} we give a brief survey on the notion of curvature measures of parallel sets with references to the literature. The measurability and continuity properties in Lemma \ref{lem:measurability} and Lemma \ref{lem: contcurvmeas} are used in the sequel for the probabilistic approach and application of the classical Renewal theorem. But they are also of independent interest, e.g.\ for possible extensions to other classes of random fractals as considered in \cite{RU11}, \cite{BHS12}, and \cite{Tr17}.

Section \ref{sec:def_main_res} contains the construction of homogeneous random fractals and the statement of the main result (Theorem \ref{maintheorem} together with Remarks~\ref{rem:partial}, \ref{rem:R} and \ref{rem:det-and-rand-rec-case}).
The proof of Theorem~\ref{maintheorem} in Section \ref{proofs} is split into several steps. As in the case of self-similar fractals and random recursive constructions the Renewal theorem is a main tool. Most effort is needed for verifying the corresponding conditions. It turns out that the integral representations for the above limits derived here for homogeneous random fractals are very similar to those known for (random) self-similar sets.

In Section \ref{sec:ex}, the hypothesis of Theorem \ref{maintheorem} is replaced in part by a simpler but stronger condition, cf.\ Proposition \ref{prop:easier-cond} and Theorem~\ref{thm:1.1}, which is easier to verify in concrete examples.
This is then demonstrated for a family of homogeneous random Sierpi\'nski gaskets as mentioned above.

In the Appendix 
the continuity of the curvature measures of parallel sets at regular distances is proved, cf.\ Theorem \ref{P_cont}, which is of independent interest. It completes, in particular, the proof of Lemma 2.3.4 in \cite{Za11} and can be applied to other deterministic and stochastic models.

\section{Curvature measures of parallel sets}\label{sec:curv parallel sets}
\subsection{Definitions and relevant properties}\label{sec:def curv}
For certain classes of compact sets $K\subset\rd$ (including many classical singular sets) it turns out that for Lebesgue-almost all distances $r>0$ the parallel set
\begin{\eq}\label{parallelset}
K(r):=\{x\in\rd: d(x,K)=\min_{y\in K}|x-y|\le r\}
\end{\eq}
possesses the property that the closure of its complement
\begin{\eq}\label{closcompl}
\widetilde{K(r)}:= \overline{K(r)^c}
\end{\eq}
is a set with positive reach with Lipschitz boundary. (Recall that $X\subset \rd$ is a set with positive reach, if for some $\delta>0$ every point $x\in X(\delta)$  has a unique point $\Pi_X(x)\in X$ nearest to $x$.)
A sufficient condition is that $r$ is a regular value of the Euclidean distance function to $K$ (see Fu \cite[Theorem 4.1]{Fu85} together with \cite[Proposition 3]{RZ03}). In this case
both sets $\widetilde{K(r)}$ and $K(r)$ are {\it Lipschitz $d$-manifolds of bounded curvature} in the sense of \cite{RZ05}, implying that their {\it k-th Lipschitz-Killing curvature measures}, $k=0,1,\ldots,d-1$, are determined in this general context and agree with the classical versions in the special cases. (See also Chapter~9 in \cite{RZ19} for some background and extensions.)
Moreover, they satisfy
\begin{\eq}  \label{reflect}
C_k(K(r),\cdot)=(-1)^{d-1-k}C_k\big(\tk,\cdot\big)\, ,
\end{\eq}
which implies that the $C_k(K(r),\cdot)$ are signed measures with finite {\it variation measures}
$C_k^{\var}(K(r),\cdot)$ and, moreover, that they have explicit integral representations which reduce to the ones in \cite{Za86b} (cf. \cite[Theorem 3]{RZ05} for the general case).
The corresponding normal cycle representation will briefly be mentioned below. We list some of the main properties of curvature measures of parallel sets which will be used repeatedly:
$2C_{d-1}(K(r),\cdot)$ agrees with $(d-1)$-{\it dimensional Hausdorff measure} $\H^{d-1}$ on the boundary $\partial K(r)$. The latter is a finite measure for all $r>0$ and all compact sets $K$. Therefore we use the notation
$$C_{d-1}(K(r),\cdot):=\frac{1}{2}\H^{d-1}(\partial K(r)\cap(\cdot))$$ in the general case. For completeness we also define
$C_d(K(r),\cdot)$ to be the {\it Lebesgue measure restricted to} $K(r)$. The {\it total curvatures} of $K(r)$ are denoted by
\begin{\eq}
C_k(K(r)):=C_k(K(r),\rd)\, ,~~ k=0,\ldots ,d\, .
\end{\eq}

By an associated Gauss-Bonnet theorem (see \cite[Theorems 2,3]{RZ03}) the {\it Gauss curvature}
$C_0(K(r))$ coincides with the {\it Euler-Poincar\'{e} characteristic} $\chi(K(r))$, whenever the curvature measure $C_0(K(r),\cdot)$ is defined as described above. For smooth boundaries $\partial K(r)$ in the differential geometric setting, $C_k(K(r))$ can be interpreted as the {\it total $k$-th order mean curvature} of $K(r)$, which is extrinsic if $d-1-k$ is odd. For $k=d-2$ it is also known as extrinsic {\it total mean curvature} of $K(r)$, and for $k=d-3$ it coincides with the intrinsic {\it total scalar curvature} of $\partial K(r)$, up to certain constants.

The curvature measures are {\it motion invariant}, i.e.,
\begin{\eq}
C_k(g(K(r)),g(\cdot))=C_k(K(r),\cdot)~~\mbox{for any Euclidean motion}~g\, ,
\end{\eq}
they are {\it homogeneous of degree} $k$, i.e.,
\begin{\eq}
C_k\big((\lambda K)(\lambda r),\lambda (\cdot)\big)=C_k\big(\lambda K(r),\lambda (\cdot)\big)=\lambda^k\, C_k(K(r),\cdot)\, ,~~\lambda>0\, ,
\end{\eq}
and \emph{locally determined}, i.e.,
\begin{\eq}
C_k(K(r),(\cdot)\cap G)=C_k(K'(r'),(\cdot)\cap G)
\end{\eq}
for any open set $G\subset\rd$ such that $K(r)\cap G = K'(r')\cap G$, where $K(r)$ and $K'(r')$ are
both parallel sets such that the closures of their complements have positive reach.

\subsection{Unit normal cycles, measurability and continuity}\label{sec: measurability and continuity}
We now summarize some facts about sets with positive reach needed in the sequel:\\
Recall that $\rea X$ of a set $X\subset \rd$ is defined as the supremum over all $r\geq 0$ such that for every point $x$ in the $r$-parallel set of $X$ there is a unique point $\Pi_Xx\in X$ nearest to $x$. The mapping $\Pi_X$ (on its domain) is called the \emph{metric projection} onto $X$. For a set $X$ of positive reach the \emph{unit normal bundle} is defined as
$$\nor X:=\{(x,n)\in \rd\times S^{d-1}: x\in X ,\, n\in \operatorname{Nor}(X,x)\}$$
where $\operatorname{Nor}(X,x)$ is the dual cone to the (convex) tangent cone of $X$ at $x$.

If additionally $\nor X\cap\r(\nor X)=\emptyset$ for the {\it normal reflection} $\r\: (x,n)\mapsto (x,-n)$, then $X$ is a $d$-dimensional Lipschitz manifold with boundary (see \cite[Proposition 3]{RZ03}).

For general $X$ with $\rea X>0$ there is an associated  rectifiable current called the {\it unit normal cycle} of $X$ which is given by
$$N_X(\vphi):=\int_{\nor X}\<a_X(x,n),\vphi(x,n)\>\H^{d-1}(d(x,n))$$
for an appropriate unit simple $(d-1)$-vector field $a_X=a_1\wedge\ldots\wedge a_{d-1}$ associated a.e. with the tangent spaces of $\nor X$ and for integrable differential $(d-1)$-forms $\vphi$. In these terms for $k\le d-1$ the curvature measure may be represented by
$$C_k(X,B)=N_X\llcorner \1_{B\times\rd}(\vphi_k)=\int_{\nor X\cap (B\times\rd)}\<a_X(x,n),\vphi_k(n)\>\H^{d-1}(d(x,n))$$
for any bounded Borel set $B\subset\rd$, where the $k$-th Lipschitz-Killing curvature form $\vphi_k$ does not depend on the points $x$ and is defined by its action on a simple $(d-1)$-vector $\eta=\eta_1\wedge\ldots\wedge\eta_{d-1}$ as follows: Let $\pi_0(y,z):=y$ and $\pi_1(y,z):=z$ be the coordinate projections in $\rd\times\rd$, $e'_1,\ldots,e'_d$ be the dual basis of the standard basis in $\rd$ and $\mathcal{O}_k$ the surface area of the $k$-dimensional unit sphere. Then we have
$$\<\eta,\vphi_k(n)\>:=\mathcal{O}_{d-k}^{-1}\sum_{\ep_i\in\{0,1\},\, \sum\ep_i=d-1-k}\<\pi_{\ep_1}\eta_1\wedge\ldots\wedge\pi_{\ep_{d-1}}\eta_{d-1}\wedge n,e'_1\wedge\ldots\wedge e'_d\>\, .
$$
Below we will use the following nice behavior of parallel sets with sufficiently large distances. The diameter of a compact set $K$ is denoted by $|K|$.
\begin{lems}\label{lems:large distances}\cite[Theorem 4.1]{Za11}
For any $R>\sqrt{2}$ and $k=0,1,\ldots,d$ there exists a constant $c_k(R)$ such that for any compact set $K\subset\mathbb{R}^d$ and any $r\ge R|K|$,
$$\rea(\widetilde{K(r)})\ge|K|\sqrt{R^2-1}\, ,$$
$\partial K(r)$ is a $(d-1)$-dimensional Lipschitz submanifold, and
$$\sup_{r\ge R|K|}\frac{C_k^{\var}(K(r),\rd)}{r^k}\le c_k(R)\, .$$
\end{lems}
(It is well-known that for compact convex sets $K$ the last two properties hold for all $r>0$, the last one is the sharper version of an isodiametric inequality with an optimal constant.)

In the general case we define the set of {\it regular pairs} of compact sets and distances by
\begin{equation}\label{regular pairs}
\Reg:=\bigg\{(r,K)\in [0,\infty)\times\mathcal{K}: \widetilde{K(r)}\in\PR ,\, \nor \widetilde{K(r)}\cap\rho(\nor\widetilde{ K(r)}) =\emptyset \bigg\}\, ,
\end{equation}
where $\K$ denotes the family of all nonempty compact subsets of $\R^d$, $\PR$ stands for the {\it space of subsets of} $\rd$ {\it with positive reach} and $\rho$ is the {\it normal reflection} given by $(x,n)\mapsto (x,-n)$.
\begin{rems}
An equivalent representation of regular pairs at positive distances is
\begin{equation} \label{rp2}
\Reg\cap ((0,\infty)\times\mathcal{K}) = \left\{ (r,K)\in (0,\infty)\times\mathcal{K}:\, r\text{ is a regular value of }d_K\right\},
\end{equation}
where $d_K:x\mapsto d(x,K)$ is the distance function of $K$, see Appendix.
\end{rems}
In Fu \cite[Theorem 4.1]{Fu85} it is shown that in space dimensions $d\le 3$ for any compact set $K$ there exists a bounded exceptional set $E$ of Lebesgue measure $0$ such that for any $r\notin E$ the set $\widetilde{K(r)}$ has positive reach and $\nor \widetilde{K(r)}\cap\rho(\nor\widetilde{ K(r)})=\emptyset$. Moreover, if $r>\sqrt{d/(2d+2)}|K|$, these two properties hold for {\it any} space dimension $d$.
It follows that in space dimensions $d\le 3$ for any compact $K$ the pair $(r,K)$ is regular for Lebesgue-a.a.\ $r$. In higher dimensions we will formulate this as a regularity condition on the random fractal sets.

For probabilistic purposes we need the following measurability properties.
Let $\F$ denote the space of all closed subsets of $\rd$ provided with the Vietoris topology (generated by the sets $\{A\in\mathfrak{F}:A\cap O\ne\emptyset\}$ and $\{A\in\mathfrak{F}:A\cap C=\emptyset\}$ for open $O$ and closed $C$) and the associated Borel $\s$-algebra $\mathfrak{B}(\F)$.
On its subspace $\K$ (of \emph{nonempty compact} sets), this topology can be metrized by the \emph{Hausdorff distance}
$$d_H(K,L):=\max\left\{\max_{x\in K}d(x,L),\max_{y\in L}d(y,K)\right\},\quad K,L\in\K,$$
and $\mathfrak{B}(\K)$ denotes the Borel $\s$-algebra on $\K$.
\begin{lems} The following assertions hold. \label{lem:measurability}
\begin{enumerate}[{\rm(i)}]
\item
$\PR\in\mathfrak{B}(\mathfrak{F})$.
\item The mapping
$(r,K)\mapsto \widetilde{K(r)}$ from $[0,\infty)\times \mathcal{K}$ to $\mathfrak{F}$
is $[\mathfrak{B}([0,\infty))\otimes \mathfrak{B}(\mathcal{K})\, ,\, \mathfrak{B}(\mathfrak{F})]$-measurable.
\item
For any bounded Borel set $B\subset\R^d$ and $k=0,\dots,d$, the mappings
$X\mapsto C_k(X,B)$ and $X\mapsto C_k^{\var}(X,B)$ from $\PR$ to $\R$ are $[\mathfrak{B}(\mathfrak{F})\cap\PR,\mathfrak{B}(\R)]$-mea- surable.
\item
For $k=0,\dots,d$, the mappings $(X,F)\mapsto C_k(X,F)$ and $(X,F)\mapsto C^{\var}_k(X,F)$ from $\PR\times\K$ to $\R$ are
$[(\mathfrak{B}(\mathfrak{F})\cap\PR)\otimes\mathfrak{B}(\K),\mathfrak{B}(\R)]$-measurable.
 \item
$\Reg\in \mathfrak{B}([0,\infty))\otimes \mathfrak{B}(\K)$.
\end{enumerate}
\end{lems}
\begin{proof}
(i) See Proposition 1.1.1 in \cite{Za86}.

(ii) It is easy to see that for the centered balls $B(R)$ of radius $R$ the mappings $(r,K)\mapsto \widetilde{K(r)}\cap B(R)$
are continuous (with respect to the Hausdorff metric). Using that $\widetilde{K(r)}=\bigcup_{R=1}^\infty
(\widetilde{K(r)}\cap B(R))$ we infer the assertion.

(iii) See Theorem 2.1.2 (i) and Theorem~6.2.2 in \cite{Za86}.

(iv) Due to (iii), the mappings $X\mapsto C_k(X,\cdot)$ and $X\mapsto C_k^{\var}(X,\cdot)$ are random signed measures on the probability space $(\PR,\mathfrak{B}(\mathfrak{F})\cap\PR,\Pr)$ for any probability measure $\Pr$ on the given space. The result then follows from Lemma~\ref{L_prod_meas}.

(v) This follows from (i), (ii), and the fact that the mapping $X\mapsto \nor X$ from $\PR$ into the space of closed subsets of $\rd\times\rd$ is measurable.
\end{proof}

In order to apply the classical Renewal theorem to the curvatures of random fractals we additionally need the following continuity property.
Recall that the set of regular values of the distance function $d_K$ is open. (This can e.g.\ be seen from \eqref{reg_point} in the Appendix.) Therefore, there exists for each regular value $r_0$ an $\ep>0$ such that for any $r\in(r_0-\ep,r_0+\ep)$, $r$ is a regular value of $d_K$, too.
\begin{lems}\label{lem: contcurvmeas}
For any $(r_0,K)\in \Reg$ with $r_0>0$ and
any $k\in\{0,\ldots,d\}$, the measures $C_k(K(r),\cdot)$ converge weakly to $C_k(K(r_0),\cdot)$ as $r\to r_0$.
\end{lems}
A detailed proof will be given in the Appendix, see Theorem~\ref{P_cont}.
\begin{rems}\label{rems: large k}
If $k=d$ or $k=d-1$, then the weak convergence in Lemma~\ref{lem: contcurvmeas} remains valid for general compact sets $K$ and almost all $r_0>0$, i.e., we need not restrict $K$ to the class $\PR$ or $r_0$ to regular values of $K$. To see this, note that $r\mapsto V_d(K(r))=\L(K(r))$ is a \emph{Kneser function} (see \cite[Lemma~5]{St76}) and that $V_{d-1}(K(r))=\H^{d-1}(\partial K(r))=\frac{d}{dr}\L(K(r))$ for all $r>0$ up to a countable set (see \cite[Corollary~2.5]{RW10}).
The continuity of $V_d(r)$ at all $r>0$ and of $V_{d-1}(K(r))$ at all $r>0$ up to a countable set follows then from the properties of Kneser functions (see \cite[Lemma 2]{St76}).

Furthermore, since both mappings $K\mapsto \L(K)$ and $K\mapsto\H^{d-1}(\partial K)$ are measurable on $\K$ (cf. \cite{Za82}), Lemma \ref{lem:measurability} (iii) and (iv) are valid for $k=d$ and $k=d-1$ with $\PR$ replaced by the larger space $\K$.
\end{rems}
\section{Construction of homogeneous random fractals and statement of the results}
\label{sec:def_main_res}
For fixed $0<r_{\min}<r_{\max}<1$ let $\Sim$ be the set of all contractive similarities $g:\rd\rightarrow\rd$ with contraction ratios $r$ such that  $r_{\min}\le r\le r_{\max}$. We equip $\Sim$ with the topology given by uniform convergence on compact sets. $\CB$ denotes the associated Borel $\s$-algebra. The space $\O_0:=\bigcup_{k=2}^\infty\Sim^k$ together with the $\s$-algebra $\F_0:=\big\{A\subset\O_0: A\cap \Sim^k\in\otimes_{i=1}^k\CB~ \mbox{for all}~k\ge 2\big\}$, and with a distribution $\P_0$ on it provide the primary probability space $[\O_0,\F_0,\P_0]$. This space is used to generate a random \emph{iterated function system} (IFS) $(f_1,\ldots, f_N)$, with a random number $N$ of mappings $f_1,\ldots, f_N$ chosen randomly from $\Sim$. Note that, by construction, $N\geq 2$.

For the definition of the homogeneous model, we need a sequence of independent and equally distributed  random IFS. Therefore, the basic probability space for the model is the product space
\begin{equation}\label{basicprobspace}
[\Omega,\F,\P]:=\bigotimes_{n=1}^\infty~[\O_0,\F_0,\P_0]
\end{equation}
and the expectation symbol $\E$ will be used for integration with respect to $\P$.

The elements of $\Omega$ are denoted by
$$\o=(\o_1,\o_2,\ldots):=((f_1^1,\ldots,f_{N^1}^1),(f_1^2,\ldots,f_{N^2}^2),\ldots)\, ,$$
and $r_i^n$ are the contraction ratios of the similarities $f_i^n$. For $f_i^1$, $r_i^1$ and $N^1$ we will often write $f_i$, $r_i$ and $N$, resp. Below we will use the measurable mapping
$$\theta: \O\rightarrow \O,\quad \theta(\o_1,\o_2,\o_3,\ldots):=(\o_2,\o_3,\ldots)\,,$$
and for a random element $X$ and $n\in\N$ we define the shifted random element $X^{(n)}$ by
\begin{align}
  \label{eq:shifted-random-elements} X^{(n)}(\o):=X(\theta^n\o), \quad \o\in\O\, ,
\end{align}
where $\theta^n$ is the $n$-fold application of the shift $\theta$.
(If it is clear from the context, the argument $\o$ will be omitted.)

Note that, for each $n\in\N$, $(f_1^n,\ldots,f_{N^n}^n)$ is a random IFS  of random length $N^{n}$ with distribution $\P_0$ representing the $n$-th construction step. For different $n$ they are independent of each other.

In the sequel we assume that the {\it Uniform Open Set Condition} (UOSC) is satisfied: there exists a nonempty bounded open set $O\subset\rd$ such that $\P_0$-a.s.\
\begin{\eq}\label{UOSC}
\bigcup_{i=1}^N f_i(O)\subset O\quad\mbox{and}\quad
f_i(O)\cap f_j(O)=\emptyset~, i\ne j\,.
\end{\eq}
Then with $\P$-probability 1 all IFS in the product space fulfill this UOSC.

The corresponding random fractal set is introduced by means of a random \tit{coding tree}:
$\Sn=\Sn(\o):=\{\s_1\ldots\s_n:1\le\s_i\le N^{i},\,i=1,\ldots,n\}$ is the set of all \tit{nodes} at level $n$ and $\Si_*:=\bigcup_{n=0}^\infty\Sn$ is the set of all nodes of the tree, where $\Si_0$ denotes the empty code at level $0$.

Recall from \eqref{eq:shifted-random-elements} that $\Si_l^{(k)}$ is defined by $\Si_l^{(k)}(\o)=\Si_l(\theta^k\o)$, $\o\in\O$. For $\s=\s_1\ldots\s_k\in\Si_k$ and $\tau=\tau_1\ldots\tau_l\in\Si_l^{(k)}$ we write $\s\tau:=
\s_1\ldots\s_k\tau_1\ldots\tau_l\in\Si_{k+l}$ for the \tit{concatenation} of these codes.
If  $\s=\s_1\ldots\s_n\in\Sn$ and $0\le k\le n$, then $\s|k:=\s_1\ldots\s_k$ denotes the \tit{restriction} to the first $k$ entries of $\s$, and $|\s|:=n$ is the \tit{length} of $\s$. For any fixed $n=1,2,\ldots$, we associate to each $\s\in\Si_n$ the same random IFS $\left(f_1^{n+1},\ldots,f_{N^{n+1}}^{n+1}\right)$.  This leads to the homogeneous structure. (In the $V$-variable case these random IFS are chosen by means of $V$ different types. Here we have $V=1$, and in the case of random recursive constructions, where $V=\infty$, for different $\s\in\Si_n$ the IFS are i.i.d.) Furthermore, we define the random mappings
$$
f_\s:=f_{\s_1}^1\circ f_{\s_2}^2\circ\dots\circ f_{\s_n}^n
$$
with contraction ratios $r_\s:=r_{\s_1}^1r_{\s_2}^2\ldots r_{\s_n}^n$.
Then the random compact set
\begin{\eq}
F=F(\o)
:=\bigcap_{n=1}^\infty\bigcup_{\s\in\Sn} f_\s(\overline{O})
\end{\eq}
is $\P$-a.s.\ determined and measurable with respect to $\CB(\K)$, the Borel $\s$-algebra  determined by the Hausdorff distance $d_H$ on the space $\K$ of nonempty compact subsets of $\rd$.
It is called the associated \tit{homogeneous random fractal}. $F$ is \tit{stochastically self-similar} in the following sense (recall from \eqref{eq:shifted-random-elements} that $F^{(n)}(\o)=F(\theta^n\o)$):
$$F=\bigcup_{i=1}^N f_i(F^{(1)})~~,~~ \P\, \text{- a.s.}~.$$
More generally, for all $n\in\N$,
\begin{\eq}\label{sss}
F=\bigcup_{\s\in\Sn}f_\s(F^{(n)})~~,~~ \P\, \text{- a.s.}~,
\end{\eq}
where the random compact set $F^{(n)}$ is independent of the random mappings\\
$\{f_\s,\,\s\in\Sn$\} and has the same distribution as $F$. We will also use the abbreviation
$$F_\s:=f_\s(F^{(|\s|)})~~ {\rm for}~~~~ \s\in\Si_*\, .$$
For a boundedness condition in the application of the Renewal theorem we will further use a formula similar to \eqref{sss} with respect to some Markov stop: Fix an arbitrary constant $R>\sqrt{2}|O|$ and define for all $0<r<R$ a random subset of codes by
\begin{\eq}\label{subtree}
\Si(r):=\{\s\in\St: R \,r_\s\le r<R\,r_{\s||\s|-1}\}\, ,
\end{\eq}
where, by convention, $r_{\s||\s|-1}=1$, if $|\s|=1$. It is convenient to set $\Si(r):=\Si_0$ for $r\geq R$. Then we have
\begin{\eq}\label{ssstop}
F=\bigcup_{\s\in\Si(r)}F_\s~~,~~ \P\, \text{- a.s.}~.
\end{\eq}
In order to formulate the main results we also need the following random sets of \emph{boundary codes}, i.e., codes $\s\in\Si(r)$ for which the parallel set 
$F_\s(r)$ has distance less than $r$ to the {\it boundary} of the first iterate $f(O):=\bigcup_{i=1}^Nf_i(O)$ of the basic open set $O$ under the random similarities:
\begin{\eq}\label{btree}
\Sb(r):= \big\{\s\in\Si(r): F_\s(r)\cap(f(O)^c)(r)\ne\emptyset\big\}\, ,
\end{\eq}
Our considerations below do not depend on the choice of the constant $R$ which is related to Lemma~\ref{lems:large distances}, (where the above $R$ corresponds to $R|O|$).

\tit{In the sequel many relationships between random elements are fulfilled only with probability $1$. We will not mention this, if it can be seen from the context.} Furthermore, the different meanings of $F=F(\o)$ as random set and $F(r)=F(\o,r)$ as parallel set of the random set $F$ will also be clear from the context.

The measurability properties of the random elements used in the sequel follow from their definitions together with Lemma \ref{lem:measurability} and Remark \ref{rems: large k}.

Recall now that $2\le N$ and assume that $\E N<\infty$. Let $D$ be the number determined by
\begin{\eq}\label{minkdim}
\E\sum_{i=1}^N r_i^D = 1\, .
\end{\eq}
Note that UOSC implies $D\le d$.
\begin{\eq}\label{eq:mu-def}
\mu(\cdot):= \E \sum_{i=1}^N \1_{(\cdot)}(|\ln r_i|)\,r_i^D
\end{\eq}
is an associated probability distribution for the logarithmic contraction ratios $r_i$ of the primary random IFS. The corresponding mean value is denoted by
\begin{\eq}
\eta:=\E\sum_{i=1}^N |\ln r_i|\,r_i^D.
\end{\eq}
For our main result we need a slightly stronger condition than UOSC \eqref{UOSC}, namely the
{\it Uniform Strong Open Set Condition} (USOSC). It is satisfied, if
\begin{\eq}\label{USOSC}
\mbox{UOSC holds for some $O$ such that}~~ \P(F\cap O\ne\emptyset)>0\, .
\end{\eq}
\begin{rems}
In the literature instead of $\P(F\cap O\ne\emptyset)>0$ the condition
$$\P(F\cap O\ne\emptyset)=1$$ has been used. By the following arguments one can see that these conditions are equivalent:
If $\P(F\cap O\ne\emptyset)>0$, there must be some $n_0\in \mathbb{N}$ such that the set $$S_0:=\{\omega: \exists~ \s\in\Si_{n_0}~~\mbox{with}~~f_\s(\overline{O})\subset O\}$$ has positive probability. (Otherwise $F$ would concentrate on the boundary of $O$.) Then the sets $S_k:=\theta^{kn_0}(S_0)$, $k=0,1,2,\ldots$, are independent and have all the same probability. Hence the Borel-Cantelli lemma implies that $\P\big(\bigcap_{n=0}^\infty\bigcup_{k\ge n}S_k\big)=1$, in particular, $\P(\bigcup_{k=0}^\infty S_k)=1$. This and UOSC lead to $\P(F\cap O\ne\emptyset)=1$.
\end{rems}
By Lemma~\ref{lems:large distances} the boundary $\partial F(r)$ of the parallel set $F(r)=F(\o,r)$ is a $(d-1)$-Lipschitz manifold and $\widetilde{F(r)}$ has positive reach for all $r\geq R$, where $R$ is some constant such that $R>\sqrt{2}|O|\geq\sqrt{2}|F|$. For $r<R$, we will use the following regularity condition. (Recall from \eqref{regular pairs} the definition of the set $\Reg$ of regular pairs.)
The random set $F$ is called {\it regular} if the measure of irregular pairs vanishes, i.e., if
\begin{equation}\label{regular SSRS}
\P\times\mathcal{L}\big(\{(\o,r)\in\Omega\times(0,\infty): (r,F)\notin \Reg\}\big)=0\, .
\end{equation}
In this case we also consider the set
$$\Reg_{*}:=\{(\o,r): (r,F_{\sigma})\in \Reg~~\mbox{for all}~ \sigma\in\Sigma_{*}\}\, $$
(Recall that $\St=\bigcup_{n=1}^\infty\Si_n$.)
and note that
\begin{equation}\label{Reg*}
\int\int\1_{(\Reg_{*})^c}(\o,r)\mathcal{L}(dr)\P(d\o)=0\, .
\end{equation}
\big(To see this, let $\mathcal{N}(\o):=\{r>0: (r,F)\notin\Reg\}$
and observe that
\begin{eqnarray*}
&&\int\int\1_{(\Reg_{*})^c}(\o,r)\mathcal{L}(dr)\P(d\o)
=\int\mathcal{L}\big(\bigcup_{\s\in\St}r_\s(\o)\mathcal{N}(\theta^{|\s|}(\o))\big)\P(d\o)\\
&\le& \sum_{n=1}^\infty\int \sum_{\s\in\Si_n}\mathcal{L}\big(r_\s(\o)\mathcal{N}(\theta^n(\o)\big)\P(d\o)
=\sum_{n=1}^\infty\int\sum_{\s\in\Si_n}r_\s(\o)\mathcal{L}(\mathcal{N}(\theta^n(\o))\P(d\o)\\
&=&\sum_{n=1}^\infty\int\sum_{\s\in\Si_n}r_\s(\o)\int\mathcal{L}(\mathcal{N}(\o'))\P(d\o')\P(d\o)=0\, ,
\end{eqnarray*}
since under the regularity condition \eqref{regular SSRS} the inner integral vanishes. Here we have used that the random sets $\mathcal{N}(\theta^n(\o))$ are independent of the events up to the step $n$ and have the same distribution as $\mathcal{N}(\o)$.\big)

Now we can formulate our main result. In the sequel, the occurring essential limits and suprema are always meant with respect to
Lebesgue-a.a.\ arguments.
\begin{thms}\label{maintheorem}
Let $k\in\{0,1,\ldots,d\}$ and let $F$ be a homogeneous random fractal satisfying the Uniform Strong Open Set Condition \eqref{USOSC} with basic set $O\subset\R^d$ and $\E N<\infty$. Let $R>\sqrt{2}|O|$. For $k\le d-2$ we additionally suppose the following:
\begin{itemize}
\item[{\rm (i)}] if $d\geq 4$, then $F$ is regular in the sense of \eqref{regular SSRS},
\end{itemize}
\begin{itemize}
\item[{\rm (ii)}]
for all $r_0\in(0,R)$,
$$\E\esup_{r_0\le r\le R}\limits\,\max_{\s\in\Sb(r)}C_k^{\var}\bigg(F(r),\partial (F_\s(r))\cap\partial\big(\bigcup_{\s'\in\Si(r),~ \s'\ne\s}F_{\s'}(r)\big)\bigg)<\infty\, ,$$
\item[{\rm (iii)}]
there is a constant $C>0$ such that with probability 1,
$$\E\bigg[\max_{\s\in\Sb(r)}\, r^{-k}\,C_k^{\var}\bigg(F(r),\partial (F_\s(r))\cap\partial\big(\bigcup_{\s'\in\Si(r),~ \s'\ne\s}F_{\s'}(r)\big)\bigg)\bigg|\sharp(\Sb(r))\bigg]\le C$$
for Lebesgue almost all $r\in(0,R]$.
\end{itemize}
Let $L>0$ and set for almost all $r>0$,
$$R_{k,L}(r):=\E C_k(F(r))-\E\sum_{i=1}^N\1_{(0,L r_i]}(r)\, C_k\big(F_i(r)\big)\,.$$
Then the following assertions hold:
\begin{itemize}
\item[{\rm (I)}] If the measure $\mu$ is non-lattice, then
$$C_{k,F}^{\fr}:=\elim_{\ep\rightarrow 0}\limits\ep^{D-k}\E\,C_k(F(\ep))=\frac{1}{\eta}\int_0^{L} r^{D-k-1}R_{k,L}(r)\, dr\, .$$

\item[{\rm (II)}] If the measure $\mu$ is lattice with constant $c$, then for almost all \ $s\in[0,c)$
  $$\lim_{n\rightarrow\infty}e^{(k-D)(s+nc)}\E\,C_k\big(F(e^{-(s+nc)})\big)=\frac{1}{\eta}\sum_{m=0}^\infty e^{(k-D)(s+mc)}R_{k,L}\big(e^{-(s+mc)}\big).
  $$
	
\item[\rm{(III)}] In general,
$$\overline{C}_{k,F}^{\fr}:=\lim_{\delta\rightarrow 0}\frac{1}{|\ln\delta|}\int_\delta^1\ep^{D-k}\E\, C_k(F(\ep))~\ep^{-1}d\ep = \frac{1}{\eta}\int_0^{L} r^{D-k-1}R_{k,L}(r)\, dr.$$
\end{itemize}
\end{thms}

\begin{rems}\label{rem:partial}
In conditions (ii) and (iii) the boundary signs $\partial$ can be omitted, since $\Int F_\sigma(r)\subset\Int F(r)$ for any $\sigma\in\St$ and the curvature measures are concentrated on the boundary of the set $F(r)$. Similarly as in the deterministic case (see \cite[Example 4.10]{Wi11}) one can construct an example of a homogeneous random fractal where these conditions are not satisfied.
\end{rems}
\begin{rems}\label{rem:R}
It is a consequence of the statement that the limit expressions in (I), (II) and (III) do not depend on the choice of the constant $L$, which gives some flexibility in applications, see Section~\ref{sec:ex}. This independence can also be seen directly using the self-similarity of $F$, cf.\ Lemma~\ref{lem:R2}. The proof of Theorem~\ref{maintheorem} will be given for the choice $L=R$, where $R$ is the constant appearing in the conditions (ii) and (iii). It is easily seen that, if these two conditions are satisfied with some $R>\sqrt{2}|O|$, then they are also satisfied with any other constant $\tilde R>\sqrt{2}|O|$ instead of $R$. (Indeed, this is obvious for $\tilde R<R$, since in this case the suprema in (ii) are taken over a smaller range of values and also the expectation in (iii) needs to be bounded for a smaller range of values $r$ only. For $\tilde R>R$, this follows from Lemma~\ref{lems:large distances}, which implies that, for almost all $R<r\leq\tilde R$, almost surely
\begin{align*}
   C_k^{\var}\bigg(F(r),\partial (F_\s(r))\cap\partial\big(\bigcup_{\s'\in\Si(r),~ \s'\ne\s}F_{\s'}(r)\big)\bigg)\leq C_k^{\var}\big(F(r)\big)\leq c_k(R)\tilde R^k.
\end{align*}
Thus, if conditions (ii) and (iii) hold with $R$, they also hold with $\tilde R>R$, the latter possibly with a larger constant $C$.)
This shows that a proof of Theorem~\ref{maintheorem} for some $L>\sqrt{2}|O|$ (e.g.\ for $L=R$) implies indeed the validity of the stated formulas for any $L>\sqrt{2}|O|$.
The argument in Lemma~\ref{lem:R2} below shows that the latter restriction can be relaxed to $L>0$.
\end{rems}
\begin{rems}\label{rem:det-and-rand-rec-case}
The special case of deterministic self-similar sets satisfying OSC is included in Theorem~\ref{maintheorem}. In this case, statement and formulas reduce to the ones obtained in \cite{Wi08,Za11} (and in \cite{Ga00} for the case $k=d$). The formulas in \cite[Thm.~2.3.6]{Wi08} are stated with $L=1$ and the ones in \cite[Thm.~2.3.8]{Za11} correspond to the choice $L=R$, but Remark~\ref{rem:R} applies similarly in these situations.

Note also that the formulas for the mean fractal curvatures are structurally equal to the ones obtained for the fractal curvatures (almost surely and in the mean) for self-similar random sets in \cite{Ga00, Za11}. Indeed, if $F$ is a homogeneous random fractal as in Theorem~\ref{maintheorem} generated by some random IFS, and $K$ denotes the random self-similar set generated by the same random IFS, then the formulas for their mean fractal curvatures coincide, except that in the integrand $R_{k,L}$ the set $F$ has to be replaced by $K$.  It turns out that in certain situations, both functions coincide, see Example~\ref{ex:homSG}.
In general, this is probably not true.
\end{rems}

\section{Proofs}\label{proofs}
The proof of Theorem \ref{maintheorem} is split into several steps. We start with the main part in which the problem is reduced to an application of the classical Renewal theorem. In order to verify the assumptions of this theorem, we will show that the function arising in the renewal equation is Lebesgue-a.e.\ continuous and bounded by a directly Riemann integrable function. For the first property condition (ii) is used and for the second one condition (iii). The required estimates are shown in a sequence of lemmas. Part of the estimates can be reduced to Lemma \ref{lems:large distances} (see Lemma \ref{second summand} below). The others follow from Lemmas \ref{lems:xik1} and \ref{lem:mainlemma}, where the latter is essential. Finally, in Lemma \ref{lem:R2} it is shown that the limit formula in the assertion of the theorem is the same for all $L>0$.

\begin{proof}[Proof of Theorem \ref{maintheorem} for $L=R$ and up to the estimates \eqref{bound} and \eqref{locbound}]
For \newline
$L>0$ the function $\xi_k^L$ is $(\P\times\mathcal{L})$-almost everywhere defined by
$$\xi_k^L(r)=\xi_k^L(\o,r):=\1_{(0,L]}(r)C_k(F(r))-\sum_{i=1}^N \1_{(0,Lr_i]}(r)C_k(F_i(r)), \quad (\o,r)\in\Reg_{*}.$$

Below we will see that the expectations of the absolute values of the two summands on the right hand side are finite. Then in view of \eqref{Reg*} the function
$$R_{k,L}=\E \xi_k^L\, $$
in Theorem~\ref{maintheorem} is determined at a.a.\ arguments. If $L=R$, where $R$ is as in Theorem~\ref{maintheorem}, we will omit the subscript $R$ and write $\xi_k(r):=\xi_k^R(r)$ and $$R_k(r)=\E \xi_k(r)\, .$$ By the motion invariance and scaling property of $C_k$, we get for $(\o,r)$ as above,
$$\xi_k(r)=\1_{(0,R]}(r)C_k(F(r))- \sum_{i=1}^N \1_{(0,R]}(r/r_i)r_i^k C_k(F^{(1)}(r/r_i))\,.$$
In order to translate the problem into the language of the Renewal theorem, we substitute $r=Re^{-t}$ and define
\begin{align*}
  Z_k(t)=Z_k(\o,t)&:=\1_{[0,\infty)}(t)e^{(k-D)t} C_k(F(Re^{-t})),\\
  z_k(t)=z_k(\o,t)&:=e^{(k-D)t}\xi_k(Re^{-t}),
\end{align*}
whenever $(\o,Re^{-t})\in\Reg_{*}$.
Note that $z_k(\o,t)=0$ for $t< 0$. We infer from the above relations that for such $(\o,t)$,
$$Z_k(\o,t)=\sum_{i=1}^N r_i^D Z_k(\theta(\o),t-|\ln r_i|)+z_k(\o,t),$$
where $(\theta(\o),t-|\ln r_i(\o)|)\in\Reg_{*}$ for $i=1,\ldots, N(\o)$.

Denote
$$T:=\left\{t>0: \left(\o,Re^{-t}\right)\in\Reg_{*}~{\rm for}~\P\text{- a.a.}~\o\right\}.$$
Let $\mu^{\ast n}$ be the $n$th convolution power of the distribution $\mu=\E\sum_{i=1}^N \1_{(\cdot)}(|\ln r_i|)r_i^D$, if $n\geq 1$, $\mu^{\ast 0}$ the Dirac measure at $0$, and
$$U(t):=\sum_{n=0}^\infty\mu^{\ast n}((0,t]),~t>0.$$
Note that the summands on the right vanish for $n>t/|\ln r_{\max}|$, so that the summation is finite for each $t>0$. (In the sequel $U$-a.a.\ means a.a.\ with respect to the corresponding measure.) Below we will show the following.
\begin{lems}\label{ts}
$t\in T$ implies $t-s\in T$ for $\mu$-a.a.\ $s\leq t$ and for $U$-a.a.\ $s\leq t$.
\end{lems}
Then we infer for $t\in T$ from the above equality for $Z_k(\o,t)$ that
\begin{align*}
\E|Z_k(t)|\leq&\int\sum_{i=1}^{N(\o)} r_i(\o)^D\left|Z_k\left(\theta(\o),t-|\ln r_i(\o)|\right)\right|\P(d\o)+\E|z_k(t)|\\
=&\int \E|Z_k(t-s)|\mu(ds)+\E|z_k(t)|\,.
\end{align*}
In view of Lemma \ref{ts}
 we obtain from iterated application of this inequality that
 $$\E|Z_k(t)|\leq\int_0^t\E|z_k(t-s)|dU(s), \quad  t\in T.$$

 Below we will show that there exist some constants $c_k>0$ and $\delta>0$ such that for all $u\in T$ we have
\begin{\eq}\label{bound}
\E|z_k(u)|\le c_k\1_{[0,\infty)}(u)e^{-\delta u}.
\end{\eq}
(In the proof for $k\le d-2$ we will use condition (iii) of Theorem~\ref{maintheorem}.) Furthermore, since in our case $U(t)\le t/|\ln r_{\max}|$ for all $t>0$, we infer from \eqref{bound} and the definition of $Z_k$ that for some constants $d_k$ and $d'_k$,
$$\E|Z_k(t)|\le \int_0^t \E|z_k(t-s)| dU(s)<d_k\,,~t\in T,$$
and consequently,
\begin{\eq}\label{expectation bound}
\E|C_k(F(r))|\leq d'_k\, r^{k-D}|\ln(r/R)|~,~~\mbox{for a.a.}~ 0<r\leq R.
\end{\eq}
This shows in particular the finiteness of the expectations mentioned at the beginning of the proof.

Moreover, we can repeat to above arguments omitting the absolute value signs and replacing the corresponding inequalities by  equalities in order to obtain the renewal equation in the sense of Feller \cite{Fe71}. We get for all $t\in T$
$$\E Z_k(t)=\int_0^t \E Z_k(t-s)\mu(ds)+\E z_k(t)$$
and
\begin{align}
   \label{eq:Z_k-repres} \E Z_k(t)=\int_0^t \E z_k(t-s)\, dU(s).
\end{align}

Below we will also show that for all $0<r_0\le R$,
\begin{\eq}\label{locbound}
\E\esup_{r>r_0}\limits|\xi_k(r)|<\infty,
\end{\eq}
where $\esup$ means the supremum over all arguments where the function is determined.
(In the proof of \eqref{locbound} for $k\le d-2$ we will use condition (ii) of the theorem.)

We now define two auxiliary functions on $(0,\infty)$ by
\begin{align*}
\overline{z}_k(t)&:=\begin{cases}
   \E z_k(t), &{\rm if}~t\in T,\\
   \limsup\limits_{\substack{t'\to t\\t'\in T}}\E z_k(t'), &{\rm if}~t\in (0,\infty)\setminus T,
\end{cases}\\
\overline{Z}_k(t)&:=\int_0^t\overline{z}_k(t-s)dU(s),~t>0.
\end{align*}
Then in view of Lemma \ref{ts} and \eqref{eq:Z_k-repres},
\begin{align}\label{Zbar}
\overline{Z}_k(t)=\E Z_k(t) \quad \text{ for } t\in T.
\end{align}
Assumption (i) and Lemma~\ref{lem: contcurvmeas} imply that the random function $\xi_k(\o,\cdot)$ is continuous in the second argument at all arguments $r$, where $(\o,r)\in\Reg_{*}$. (For $k=d,d-1$ we do not need (i) for this conclusion.)
The dominated convergence theorem (justified by \eqref{locbound}) yields that at each $t\in T$ the function $\overline{z}_k$, which agrees with $e^{(k-D)t}\E \xi_k(Re^{-t})$ at such $t$, is continuous, i.e., $\overline{z}_k$ is continuous Lebesgue-a.e.. Moreover, in view of \eqref{bound}, $\overline{z}_k$ is bounded by a directly Riemann integrable function. Thus, according to  Asmussen \cite[Prop. 4.1, p. 118]{As87}, $\overline{z}_k$ is directly Riemann integrable, too. Therefore the classical Renewal theorem in Feller \cite[p.\ 363]{Fe71} can be applied, which yields that
$$\lim_{t\rightarrow\infty}\overline{Z}_k(t)=\frac{1}{\eta}\int_0^\infty\overline{z}_k(t)\,dt.$$
The right hand side agrees with
$$\frac{1}{\eta}\int_0^\infty\E z_k(t)\,dt=\frac{1}{\eta}\int_0^\infty e^{(k-D)t}R_k(Re^{-t})\, dt,$$
since $\overline{z}_k(t)=\E z_k(t)$ for Lebesgue-a.a.\ $t$.
Multiplying $\j^{D-k}$ in this equation and substituting $r=R e^{-t}$ under the integral, assertion (I) follows in view of
\eqref{Zbar}.

Since $\E Z_k$ is bounded on finite intervals, in the non-lattice case the corresponding average limit in (III) is a consequence.

In the lattice case, the Renewal theorem provides the limit along arithmetic progressions with respect to the lattice constant, here only for those sequences along which the function is determined. This shows assertion (II). This also implies the average convergence (III). For more details we refer to the arguments of Gatzouras at the end of the proof of \cite[Theorem~2.3]{Ga00} in the classical case.
\end{proof}
\begin{proof}[Proof of Lemma \ref{ts}]
For $t\in T$ we have
\begin{align*}
1&=\P\left(\left\{\o: (Re^{-t}, F_\s(\o))\in\Reg~\mbox{for all}~\s\in\St(\o)\right\}\right)\\
&\le\P\left(\left\{\o: (Re^{-t},F_{i\tau}(\o))\in\Reg \text{ for } i=1,\ldots,N(\o)\,,~\tau\in\St(\theta(\o))\right\}\right).
\end{align*}
Using that $(Re^{-t},F_{i\tau}(\o))\in\Reg$ if and only if $(Re^{-(t-|\ln r_i(\o)|)},F_{\tau}(\theta(\o)))\in\Reg$ and the product structure of the basis probability space, we infer
\begin{align*}
1=\int\P\Big(\Big\{\o' :   \big(R&e^{-(t-|\ln r_i(\o)|)}, F_{\tau}(\o')\big)\in\Reg \text{ for }\\
&\qquad\quad i=1,\ldots,N(\o),\tau\in\St(\o')\Big\}\Big)\P(d\o)\,.
\end{align*}
 Hence, we get for $\P$-a.a. $\o$,
$$\P\left(\left\{\o': \left(Re^{-(t-|\ln r_i(\o)|)}, F_{\tau}(\o')\right)\in\Reg \text{ for } \tau\in\St(\o')\right\}\right)=1\,,$$
for $i=1,\ldots,N(\o)$,
and therefore
$$\int\sum_{i=1}^N r_i(\o)^D\P\left(\left\{\o': \left(Re^{-(t-|\ln r_i(\o)|)}, F_{\tau}(\o')\right)\in\Reg \text{ for } \tau\in\St(\o')\right\}\right)\P(d\o)=1,$$
since $\int\sum_{i=1}^N r_i(\o)^D\P(d\o)=1$.
By the definition of the measure $\mu$, this means that
$$\int\P\left(\left\{\o': \left(Re^{-(t-s)}, F_{\tau}(\o')\right)\in\Reg \text{ for } \tau\in\St(\o')\right\}\right)\mu(ds)=1\,.$$
Thus, for $\mu$-a.a. $s\le t$,
$$\P\left(\left\{\o': \left(Re^{-(t-s)}, F_{\tau}(\o')\right)\in\Reg \text{ for } \tau\in\St(\o')\right\}\right)=1,$$
i.e., for $t\in T$ and $\mu$-a.a. $s\le t$ we get $t-s\in T$. Iterated application of this result yields that $t\in T$ implies $t-s\in T$ for $U$-a.a. $s\le t$.
\end{proof}
In order to complete the proof of Theorem~\ref{maintheorem}, it remains to verify \eqref{bound} and \eqref{locbound}.
By the definition of $z_k$, the first one is equivalent to showing that there are constants $c_k'$ and $\delta>0$ such that
\begin{\eq} \label{eq:bound-1}
\E\big|\xi_k(r)\big|\le c_k'\, r^{k-D+\delta}~~\mbox{for a.a.}~r>0\, .
\end{\eq}
 (In particular, if this inequality holds, then \eqref{bound} is satisfied with the same $\delta$ and  $c_k=\j^{k-D+\delta}c_k'$.) Observe that, by Lemma~\ref{lems:large distances}, \eqref{locbound} and \eqref{eq:bound-1} are satisfied for all $r_0>R$ and $r>R$, respectively. In order to prepare the estimates for a.a. $r\leq R$ note first that for a.a.\ $r\in(0,R]$ we get
\begin{align} \label{eq:bound-2}
  \notag
	|\xi_k(r)|&=\bigg| C_k(F(r))
-\sum_{i=1}^N(1-\1_{(\j r_i,\j]}(r))C_k(F_i(r))\bigg|\\
\notag &= \bigg|C_k(F(r))
-\sum_{i=1}^N C_k(F_i(r))+\sum_{i=1}^N \1_{(Rr_i,R]}(r)C_k(F_i(r))\bigg|\\
\notag &\leq \bigg|C_k(F(r))
-\sum_{i=1}^N C_k(F_i(r))\bigg|
+\bigg|\sum_{i=1}^N\1_{(Rr_i,R]}(r)C_k(F_i(r))\bigg|\\
\notag & =:\xi_{k1}(r)+\xi_{k2}(r)\, .
\end{align}
Therefore, it suffices to provide the required bounds for the two summands in the last line instead of $|\xi_k(r)|$. For the second summand the estimates \eqref{locbound} and \eqref{eq:bound-1} follow from the next statement, since we supposed that $\E N<\infty$.

\begin{lems}\label{second summand}
 There exists some constant $b_k$ such that for all $r\in(0,\j]$,
 $$\xi_{k2}(r)\leq b_k Nr^k ,~ a.s.\, .$$
 \end{lems}

\begin{proof}
   Recalling that $F_i(r)=(f_i(F^{(1)}))(r)=f_i(F^{(1)}(r/r_i))$, we have
   \begin{align*}
    \xi_{k2}(r)= \left|\sum_{i=1}^N\1_{(\j r_i,\j]}(r) C_k(F_i(r))\right|
   &\leq \left|\sum_{i=1}^N \1_{(\j r_i,\j]}(r) C_k^{\var}(f_i(F^{(1)}(r/r_i)))\right|\\
   &=\sum_{i=1}^N \1_{(\j r_i,\j]}(r)r_i^k C_k^{\var}(F^{(1)}(r/r_i)).
   \end{align*}
By the choice of $R$, we have for any $r\in (\j r_i,\j]$ that
$r/r_i>R>\sqrt{2}|O|\geq \sqrt{2}|F^{(1)}|$.
 Hence, we can apply Lemma \ref{lems:large distances} and infer that there is a constant $b_k=b_k(\j)$ such that almost surely
$$C_k^{\var}(F^{(1)}(r/r_i)) \leq  b_k (r/r_i)^k$$
for all $r>\j r_i$ and all $i$. Plugging this into the above estimates we get for all $r\le R$,
$$\left|\sum_{i=1}^N\1_{(\j r_i,\j]}(r) C_k(F_i(r))\right|\leq N b_kr^k\, ,$$
i.e., the assertion.
\end{proof}
It remains to prove the required bounds for the first summand $\xi_{k1}$ in the above estimates.
\begin{lems}\label{lems:xik1}
There exist constants $c_k'>0$ and $\delta>0$ such that
\begin{\eq}\label{auxbound-2}
\E\xi_{k1}(r)=\E\big|C_k(F(r))-\sum_{i=1}^N C_k(F_i(r))\big|\le c_k'\, r^{k-D+\delta}\,,~\mbox{for a.a.}~r\in(0,R],
\end{\eq}
and
\begin{\eq}\label{auxbound-3}
\E\esup_{r_0<r\le R}\limits \xi_{k1}(r)<\infty,\quad 0<r_0<R .
\end{\eq}
\end{lems}

\begin{proof}
For $k\in\{d-1,d\}$ the estimate \eqref{auxbound-3} is a simple consequence of the facts that for any compact $K$ and $r>0$ we have
$C_d(K(r))\le\const(2r+|K|)^d$ and $C_{d-1}(K(r))\le\const r^{-1}C_d(K(r))$. The latter follows from  the Kneser property of the volume function, i.e., $\frac{\partial}{\partial r}\L(K(r))\le\frac{d}{r}\L(K(r))$ for almost all $r$ (see e.g.\cite[Lemma 4.6 and its proof]{RSS09}), together with \cite[Corollary 2.6]{RW10}.

For $k\le d-2$ we decompose the total $k$th curvatures by means of the corresponding curvature measures:
$$
C_k(F(r))  = C_k(F(r),A_r)+C_k\big(F(r), (A_r)^c\big) ,
$$
where
\begin{equation}
A_r:=\bigcup_{j\ne k} f_j(\co)(r)\cap f_k(\co)(r)\, .
\end{equation}
Similarly,
$$C_k(F_i(r))= C_k(F_i(r),A_r)+C_k\big(F_i(r), (A_r)^c\big)\, ,\quad i=1,\ldots,N\,. $$
The local definiteness of the curvature measure $C_k$ implies
$$C_k(F_i(r),(A_r)^c)=C_k(F_i(r),B^i)=C_k(F(r),B^i)$$
and $F(r)\cap(A_r)^c$ is the disjoint union of the sets $B^i:=F_i(r)\setminus A_r$,\, $i=1,\ldots,N$. Hence,
$$
C_k(F(r),(A_r)^c)-\sum_{i=1}^N C_k(F_i(r),(A_r)^c)=0\, .$$
Substituting this in the above expression for $\xi_{k1}$, we infer that
$$\xi_{k1}(r)=\big| C_k(F(r),A_r)-\sum_{i=1}^N C_k(F_i(r),A_r)\big|$$
and so, by the scaling property of $C_k$,
\begin{eqnarray*}
\xi_{k1}(r) & = &\bigg| C_k(F(r),A_r)-\sum_{i=1}^N r^k_i\, C_k\big(F^{(1)}\big(\frac{r}{r_i}\big),f_i^{-1}(A_r)\big)\bigg|\\
& = & \big|\xi_{k3}(r)-\xi_{k4}(r)-\xi_{k5}(r)\big|\leq \left|\xi_{k3}(r)\right|+\left|\xi_{k4}(r)\right|+\left|\xi_{k5}(r)\right|
\end{eqnarray*}
where
\begin{eqnarray*}
\xi_{k3}(r)& := & C_k(F(r),A_r),\\
\xi_{k4}(r)& := & \sum_{i=1}^N r_i^k\, \1_{(0,\j]}\big(\frac{r}{r_i}\big)\, C_k\bigg(F^{(1)}\big(\frac{r}{r_i}\big),f_i^{-1}(A_r)\bigg),\\
\xi_{k5}(r)& := & \sum_{i=1}^N r_i^k\, \1_{(\j,\infty)}\big(\frac{r}{r_i}\big)\, C_k\bigg(F^{(1)}\big(\frac{r}{r_i}\big),f_i^{-1}(A_r)\bigg)\, .
\end{eqnarray*}
Therefore, instead of proving the estimates \eqref{auxbound-2} and \eqref{auxbound-3} for $\xi_{k1}$, it suffices to prove corresponding estimates for $|\xi_{k3}|$, $|\xi_{k4}|$ and $|\xi_{k5}|$ separately.

The arguments for $|\xi_{k5}|$ are the same as for $\xi_{k2}$  taking into account that for any Borel set $B$, $|C_k(F^{(1)}(r),B)|\le C_k^{\var}(F^{(1)}(r),\R^d)$ and applying Lemma \ref{lems:large distances}.

For estimating $\xi_{k3}$ and $\xi_{k4}$ we will use the set inclusions
$$A_r\subset (f(O))^c(r)\, ,\quad f_i^{-1}(A_r)\subset O^c\big(\frac{r}{r_i})\, ,\quad {\rm and}~~ O^c(r)\subset f(O)^c(r)\, .$$
(Recall that $f(O)=\bigcup_{i=1}^Nf_i(O)$ and $O$ is the open set from UOSC.) Then for $|\xi_{k3}|$ the estimates \eqref{auxbound-2} and \eqref{auxbound-3}  follow from \eqref{auxbound-6} and \eqref{auxbound-5}, respectively, in Lemma~\ref{lem:mainlemma} below.

Furthermore, for Lebesgue-a.a.\ $r\in(r_0,R]$, we obtain from the above set inclusions
\begin{eqnarray*}
|\xi_{k4}(r)|
&\le&\sum_{i=1}^N \1_{(0,R]}(\frac{r}{r_i})r_i^kC_k^{\var}\big(F^{(1)}(\frac{r}{r_i}),O^c(\frac{r}{r_i})\big)\\
&\le& N\esup_{r_0<r\le R}\limits C_k^{\var}\big(F^{(1)}(r),O^c(r)\big)\, .
\end{eqnarray*}
Since the random set $F^{(1)}$ is independent of the events in the first step, in particular of $N$, and has the same distribution as $F$ we infer
$$\E\esup_{r_0<r\le R}\limits |\xi_{k4}(r)|\le \E N\, \E \esup_{r_0<r\le R}\limits C_k^{\var}\big(F(r),O^c(r)\big)\, .$$
As $O^c(r)\subset f(O)^c(r)$, the estimate \eqref{auxbound-5} in Lemma~\ref{lem:mainlemma} below yields \eqref{auxbound-3} for $\xi_{k4}$.

Similarly we get
\begin{align*}
\esup_{r_0<r\le R}\limits&\E\left[\left(\frac{|\xi_{k4}(r)|}{r^{k-D+\delta}}\right)\right]\\
&\le \esup_{r_0<r\le R}\limits\E\left[\sum_{i=1}^N r_i^{D-\delta}\1_{(0,R]}(\frac{r}{r_i})
\left( \frac{r_i^{k-D+\delta}}{r^{k-D+\delta}}~ C_k^{\var}\bigg(F^{(1)}\big(\frac{r}{r_i}\big), O^c\big(\frac{r}{r_i}\big)\bigg)\right)\right]\\
&=\esup_{r_0<r\le R}\limits\E\left[\sum_{i=1}^N r_i^{D-\delta}\1_{(0,R]}(\frac{r}{r_i})
\left( \frac{r_i^{k-D+\delta}}{r^{k-D+\delta}}~\E C_k^{\var}\bigg(F^{(1)}\big(\frac{r}{r_i}\big), O^c\big(\frac{r}{r_i}\big)\bigg)\right)\right]\\
&\le\E\left(\sum_{i=1}^N r_i^{D-\delta}\right)
\esup_{0<r\le R}\limits\left(\frac{\E C_k^{\var}\big(F^{(1)}(r), O^c(r)\big)}{r^{k-D+\delta}}\right)\\
&\le \E N~ \esup_{0<r\le R}\limits\left(\frac{\E C_k^{\var}\big(F(r), f(O)^c(r)\big)}{r^{k-D+\delta}}\right)
<\infty
\end{align*}
for some $0<\delta<D$ according to \eqref{auxbound-6} in Lemma \ref{lem:mainlemma} below.
(Via conditional expectation the inner expectations are chosen with respect to the random set $F^{(1)}$. Recall that the latter is independent of the contraction ratios $r_1,\ldots,r_N$ and has the same distribution as $F$.)
This yields the remaining estimate \eqref{auxbound-2} for $\xi_{k4}$.
Thus, Lemma~\ref{lem:mainlemma} completes the proof of Theorem~\ref{maintheorem}.
\end{proof}

\begin{lems}\label{lem:mainlemma} Under the conditions of Theorem~\ref{maintheorem} we have
\begin{\eq}\label{auxbound-5}
\E\esup _{r_0<r<R}\limits C_k^{\var}\big(F(r),f(O)^c(r)\big)<\infty
\end{\eq}
for all $0<r_0<R$, and
\begin{\eq}\label{auxbound-6}
\esup _{0<r<R}\limits\left(\frac{\E C_k^{{\var}}\big(F(r),f(O)^c(r)\big)}{r^{k-D+\delta}}\right)<\infty
\end{\eq}
for some $0<\delta<D$.
\end{lems}

\begin{proof}
We start similarly as in the above  proof choosing the subtree  $\Si(r)$ from \eqref{subtree} instead of $\Si_1$ in the decomposition of the curvature measures.
First note that by a simple volume comparing argument the Uniform Open Set Condition \eqref{UOSC} implies that for all $r>r_0>0$,
\begin{\eq}\label{finite number}
\sharp\big(\Si(r)\big)\le \const (r_0)^{-d}\, ,
\end{\eq}
where $\sharp$ denotes the number of elements of a finite set.

Next recall that $F(r)=\bigcup_{\s\in\Si(r)}F_\s(r)$
for any $r>0$. Since curvature measures are locally defined, this implies
\begin{eqnarray}\label{auxestimate}
\nonumber C_k^{\var}\big(F(r), f(O)^c(r)\big)&=&C_k^{\var}\bigg(F(r), \big(\bigcup_{\s\in\Si(r)}F_\s(r)\big)\cap f(O)^c(r)\bigg)\\
&\le&\sum_{\s\in\Si_b(r)}C_k^{\var}(F(r), F_\s(r))\, ,
\end{eqnarray}
where the boundary code set $\Si_b(r)$ was defined in \eqref{btree}. Note that $\Int F_\sigma(r)\subset\INt F(r)$ for any $\sigma\in\Sigma(r)$, thus $C_k^{\var}(F(r),\Int F_\sigma(r))=0$ for $k\le d-1$ (since the curvature measures are concentrated on the boundary of $F(r)$). Then for $k\in\{d-1,d\}$ we can use the (in)equality \begin{equation}\label{auxFs}
C_k^{\var}\big(F(r),F_\s(r)\big)\le C_k\big(F_\s(r)\big)\, ,~~\s\in\Si_b(r)\, .
\end{equation}
Similarly we obtain for any $\s\in\Si_b(r)$ and $k\le d-2$,
\begin{align*}
&C_k^{\var}(F(r),F_\s(r))\\
&=C_k^{\var}\bigg(F(r),F_\s(r)\setminus\bigcup_{\s'\in\Si(r),\, \s'\ne\s}\hspace{-3mm}F_{\s'}(r)\bigg)
+C_k^{\var}\bigg(F(r),F_\s(r)\cap\bigcup_{\s'\in\Si(r),\, \s'\ne\s}\hspace{-3mm}F_{\s'}(r)\bigg)\\
&\le C_k^{\var}(F_\s(r))+C_k^{\var}\bigg(F(r),\partial F_\s(r)\cap\partial\big(\bigcup_{\s'\in\Si(r),\, \s'\ne\s}\hspace{-3mm}F_{\s'}(r)\big)\bigg)\\
&=r^k_{\s}\,  C_k^{\var}\bigg(F^{(|\s|)}\big(\frac{r}{r_{\s}}\big)\bigg)+ C_k^{\var}\bigg(F(r),\partial F_\s(r)\cap\partial\big(\bigcup_{\s'\in\Si(r),\, \s'\ne\s}\hspace{-3mm}F_{\s'}(r)\big)\bigg)\\
&=: S_1(r,\s)+S_2(r,\s)\, .
\end{align*}
Moreover, for $\s\in\Si_b(r)$ we have $\frac{r}{r_\s}\ge R$.
Hence, the first summand $S_1(r,\s)$ on the right hand side is bounded by a constant in view of Lemma \ref{lems:large distances}, since $|O|\ge|F^{(|\s|)}|$. (This holds also for $k\in\{d-1,d\}$.)
Therefore, these estimates together with \eqref{auxestimate} and \eqref{finite number} lead to
\begin{align*}
\E\esup_{r_0<r\le R}\limits& C_k^{\var}\big(F(r), f(\co)^c(r)\big)\\
&\le \E\esup_{r_0<r\le R}\limits\sum_{\s\in\Sb(r)}S_1(r,\s)
+\E\esup_{r_0<r\le R}\limits\sum_{\s\in\Sb(r)}S_2(r,\s)\\
&\le\const+\const\E\esup_{r_0<r\le R}\limits\max_{\s\in\Sb(r)}S_2(r,\s)\, .
\end{align*}
In the last estimate we have used that $\sharp(\Sb(r))$ is bounded by a constant (depending on $r_0$) according to \eqref{finite number}. Since the last summand is finite by assumption (ii) in our theorem, we obtain the first assertion \eqref{auxbound-5}.

To prove \eqref{auxbound-6}, we argue similarly that
for $0<\delta<D$,
\begin{align*}
&r^{D-\delta-k}\E C_k^{\var}\big(F(r), f(O)^c(r)\big)\\
&\quad\le\E\bigg( r^{D-\delta}\sum_{\s\in\Sb(r)}r^{-k}S_1(r,\s)\bigg)
+\E\bigg( r^{D-\delta}\sum_{\s\in\Sb(r)}r^{-k}S_2(r,\s)\bigg)\\
&\quad\le\E\big( r^{D-\delta}\sharp(\Sb(r))\max_{\s\in\Sb(r)}r^{-k}S_1(r,\s)\big)
+\E\big( r^{D-\delta}\sharp(\Sb(r))\max_{\s\in\Sb(r)}r^{-k}S_2(r,\s)\big)\, .
\end{align*}
Since $\frac{r}{r_\s}\ge R$ for $\s\in\Sb(r)$ and any $0<r\le R$, by Lemma \ref{lems:large distances}, the maximum in the first summand is uniformly bounded by some constant $c=c(R)$ independent of $r$.
Therefore, for all $r\in(0,R]$, the last sum does not exceed
\begin{align*}
c\,&\E\big[ r^{D-\delta}\sharp(\Sb(r))\big]
+\E\big[ r^{D-\delta}\sharp(\Sb(r))\max_{\s\in\Sb(r)}r^{-k}S_2(r,\s)\big]\\
&=c\,\E\big[ r^{D-\delta}\sharp(\Sb(r))\big]
+\E\bigg[ r^{D-\delta}\sharp(\Sb(r))\E\big[\max_{\s\in\Sb(r)}r^{-k}S_2(r,\s)\big)\big|\sharp(\Sb(r))\big]\bigg]\\
&\le c\,\E\big[r^{D-\delta}\sharp(\Sb(r))\big]
+\E\bigg[ r^{D-\delta}\sharp(\Sb(r))\esup_{0<r\le R}\limits \E\big[\max_{\s\in\Sb(r)}r^{-k}S_2(r,\s)\big)\big|\sharp(\Sb(r))\big]\bigg]\\
&\le (c+C)\,\E\big[ r^{D-\delta}\sharp(\Sb(r))\big]\, ,
\end{align*}
since, by condition (iii) in Theorem~\ref{maintheorem}, the supremum of the conditional expectations in the second summand is bounded by $C$. Therefore, it remains to show that for some $0<\delta<D$,
\begin{equation}\label{bnumber}
\sup_{0<r\le R}\limits\E\big[r^{D-\delta}\sharp(\Sb(r))\big]<\infty\, .
\end{equation}
This was already proved in \cite{Za20}. For convenience of the reader we replicate the arguments here.
To this aim we will use USOSC, i.e., UOSC with $O$ such that $\P(F\cap O\ne\emptyset)>0$, which implies that there exist some constants $\alpha>0$ and $0<\r<1$ such that
\begin{\eq}\label{rhoalpha}
\P(\Si(\r,\alpha)\ne\emptyset)>0~~{\rm for}~~
\Si(\r,\alpha):=\{\tau\in\Si(\r): d(x,\partial O)>\alpha\,,\,x\in f_\tau(\overline{O})\}\, .
\end{\eq}
Since $\Si(\r)$ is a Markov stop, one infers from $\E\sum_{i=1}^N r_i^D=1$ that\\ $\E\sum_{\s\in\Si(\r)}r_{\s}^D=1$
(see e.g. \cite[Proposition 1]{Za20}).
Then let $\delta$ be determined by
\begin{\eq}\label{delta}
\E\bigg(\sum_{\tau\in\Si(\r)\setminus\Si(\r,\alpha)} r_\tau^{D-\delta}\bigg)=1\, .
\end{\eq}
We next choose for all $r>0$,
\begin{\eq}\label{es}
r^*:=2R(\alpha r_{\min})^{-1}r \,.
\end{\eq}
Then we get for $i=1,\ldots,N$ and $i\s\in\Si(r^*)$ with $\s=\tau\s'$ for some $\tau\in \Si^{(1)}(\r,\alpha)$ that
$$f_{i\s}(F^{(|i\s|)})(r)\cap f(O)^c(r)=\emptyset\, .$$
(To see this note that for any $x\in f_{i\s}(F^{|i\s|})(r)$ there exists a $y\in f_{i\s}(F^{(|i\s|)})$ such that $|x-y|\le r$. Furthermore, $y\in f_{i\s}(F^{(|i\s|)})\subset f_i(F^{(1)})\subset f_i(\overline{O})\subset f(\overline{O})$, and $d(y,\partial f(O))\ge d(y,\partial f_{i\s}(O))>r_{i\s}\alpha>R^{-1}r^* r_{\min}\alpha=2r$. Consequently, $d(x,\partial f(O)^c)\ge d(y,\partial f(O)^c)-|x-y|>2r-r=r$, i.e., $x\notin (f(O)^c)(r)$.)

From this we obtain
$$\sharp(\Sb(r))\le\sum_{i=1}^N \sharp(\{w\in\Si(r^*): w=i\s,\, \s\in\Xi^{(1)}(r^*/r_i)\})\, ,$$
where the random sets $\Xi(r)$, $r>0$, are defined as
$$\Xi(r):=\Si(r)\setminus\{\s\in \Si(r): \s=\tau\s'~~\mbox{for some}~~\tau\in\Si(\r,\alpha)\}\, .$$
In these notations we get
\begin{eqnarray*}
r^{D-\delta}\E\sharp(\Sb(r))&\le&r^{D-\delta}\E\sum_{i=1}^N\sharp(\Xi^{(1)}(r^*/r_i))\\
&=&(\alpha r_{\min}/2)^{D-d}\E\sum_{i=1}^N r_i^{D-\delta}(r^*/r_i)^{D-\delta}\sharp(\Xi^{(1)}(r^*/r_i))\\
&=& (\alpha r_{\min}/2)^{D-d} \E\sum_{i=1}^N r_i^{D-\delta}(r^*/r_i)^{D-\delta}\E\sharp(\Xi^{(1)}(r^*/r_i))\\
&=& \const\E\sum_{i=1}^N r_i^{D-\delta}\psi(r^*/r_i)\,,
\end{eqnarray*}
where we have used that $\Xi^{(1)}$ is independent of the events in the first step and has the same distribution as $\Xi$ and then the notation $\psi(r):=r^{D-\delta}\E\sharp(\Xi(r))$. Now it suffices to show that the function $\psi$ is bounded.

Similarly as above, using \eqref{rhoalpha} and the definition of $\Xi(r)$ we infer for sufficiently large $M$ and $r<\rho$,
\begin{eqnarray*}
\psi(r)&=&\E\sum_{\tau\in\Si(\r)\setminus\Si(\r,\alpha)}r_\tau^{D-\delta} (r/r_\tau)^{D-\delta}\sharp(\Xi^{(|\tau|)}(r/r_\tau)\\
&=&\sum_{n=1}^M\E\sum_{\stackrel{\tau\in\Si(\r)\setminus\Si(\r,\alpha)}{|\tau|=n}}r_\tau^{D-\delta}(r/r_\tau)^{D-\delta}\sharp(\Xi^{(n)}(r/r_\tau)\\
&=&\sum_{n=1}^M\E\sum_{\stackrel{\tau\in\Si(\r)\setminus\Si(\r,\alpha)}{|\tau|=n}} r_\tau^{D-\delta} \psi(r/r_\tau)
=\E\sum_{\tau\in\Si(\r)\setminus\Si(\r,\alpha)} r_\tau^{D-\delta}\psi(r/r_\tau)\\
&\le&\E\sum_{\tau\in\Si(\r)\setminus\Si(\r,\alpha)}r_\tau^{D-\delta}\esup_{r'\ge r/\r}\limits\psi(r')=\esup_{r'\ge r/\r}\limits\psi(r')\, ,
\end{eqnarray*}
where we have used that the random sets $\Xi^{(n)}(r)$ are independent of the behavior of the system up to the step $n$ via conditional expectation, that they have the same distribution as $\Xi(r)$, and then \eqref{delta}. Hence, $\psi(r)\le\esup_{r'\ge r/\r}\limits\psi(r')$ for any $r<\r$ which implies
$$\sup_{r\ge\r^{k+1}}\psi(r)\le \sup_{r\ge\r^k}\psi(r)~~\mbox{for all}~k\, .$$
Since the function $\psi$ is bounded on any finite interval away from zero it is bounded on $(0,R)$. This completes the proof of \eqref{bnumber}.
\end{proof}

We have proved Theorem~\ref{maintheorem} for the special case $L=R$ (and it was argued in Remark~\ref{rem:R} that this implies, it holds for any $L>\sqrt{2}|O|$.) In order to complete the proof, we need to show that the assertions hold also for $L\in(0,\sqrt{2}|O|]$. Recall that $R_k=R_{k,R}$.
\begin{lems}\label{lem:R2}
Under the conditions of Theorem \ref{maintheorem} for any $0<L<R$,
\begin{align*}
   \int_0^L r^{D-k-1}R_{k,L}(r) dr
  &=\int_0^{R} r^{D-k-1}R_k(r)dr.
\end{align*}
\end{lems}
\begin{proof} Observe that
\begin{eqnarray*} \label{eq:L-R}
  \int_0^L r^{D-k-1}R_{k,L}(r) dr
  =\int_0^R r^{D-k-1}R_k(r)dr-\int_L^R r^{D-k-1}\E C_k(F(r)) dr\\
  + \int_0^\infty r^{D-k-1} \E\sum_{i=1}^N \1_{(Lr_i,Rr_i]}(r) C_k(F_i(r)) dr =:I_1-I_2+I_3\, ,
\notag \\	
\end{eqnarray*}
provided that all three integrals on the right side converge absolutely. For the first integral $I_1$  this is included in the proof of Theorem \ref{maintheorem} for $L=R$, and for $I_2$ it follows from the estimate \eqref{expectation bound}. Replacing in $I_3$ the curvature measure $C_k$ by $|C_k|$ and
interchanging expectation and integration we obtain
\begin{align*}
\int_0^\infty r^{D-k-1}& \E\sum_{i=1}^N \1_{(Lr_i,Rr_i]}(r) |C_k(F_i(r))| dr\\
&= \E\sum_{i=1}^N \int_0^\infty r^{D-k-1} \1_{(Lr_i,Rr_i]}(r)|C_k(F_i(r))| dr\\
&=\E\sum_{i=1}^N \int_0^\infty r^{D-k-1} \1_{(L,R]}(r/r_i) r_i^k |C_k(F(r/r_i))| dr\\
&=\E\sum_{i=1}^N r_i^{D} \int_L^R \bar r^{D-k-1} |C_k(F^{(1)}(\bar r))| d\bar r\\
&=\E\sum_{i=1}^N r_i^D \E \int_L^R \bar r^{D-k-1}|C_k(F^{(1)}(\bar r))| d\bar r\\
&=\int_L^R \bar r^{D-k-1} \E |C_k(F(\bar r))| d\bar r,
\end{align*}
where we have substituted $\bar r=r/r_i$ in the second step and used that $F_i=f_i(F^{(1)})$. For the third step notice that $F^{(1)}$ is independent of the first step of the construction. Hence the expectation can be written  as a product of two expectations, where the first one $\E\sum_{i=1}^N r_i^{D}$ equals 1, by the definition of $D$. In the second one we interchanged again  expectation and integration. The last integral is finite because of \eqref{expectation bound}. Hence, we have shown the existence of the integral $I_3$. Now we can repeat the last transformations omitting the absolute value signs in order to get  $I_3=I_2$. Together with the above decomposition this proves the assertion.
\end{proof}

\section{Sufficient conditions and examples} \label{sec:ex}

We discuss some examples to illustrate our main result and compare it with the known results in the random recursive case.
In order to simplify the verification of the conditions (ii) and (iii) in Theorem~\ref{maintheorem}, we discuss first some simpler sufficient conditions, which may equally be used in the random recursive case. Recall that $O_\s=f_\s(O)$ for any word $\s\in\Sigma_{*}$. As a first step, the following observation clarifies that not too many of the parallel sets $O_\s(r)$,  $\s\in\Sigma(r)$ intersect. Recall also that we assume here a uniform lower bound $r_{\min}$ for the contraction ratios, cf.\ the first lines of Section~\ref{sec:def_main_res}. For a similar estimate in the deterministic case see e.g.~\cite[Lemma~5.3.1]{Wi08}.

\begin{lems} \label{lem:Gamma-bound}
  Let $F$ be a homogeneous random fractal or a random self-similar set satisfying UOSC~\eqref{UOSC} for some open set $O$. Then there is a constant $\Gamma>0$ such that $\P$-a.s.\ for all $r>0$ and all $\s\in\Sigma(r)$,
  \begin{align*}
    \sharp\{\s'\in\Sigma(r): O_\sigma(r)\cap O_{\sigma'}(r)\neq\emptyset\}\leq\Gamma\,.
  \end{align*}
\end{lems}
\begin{proof}
We provide a proof here for the homogeneous model. The one for the random recursive model is literally the same, when $\P$ is replaced by the corresponding measure, see, e.g., \cite[eq.~(9)]{Za11}.  Let $O$ be the open set in UOSC. The condition implies that $\P$-a.s.\ for any $r>0$ the family $\{O_\s: \s\in\Sigma(r)\}$ consists of pairwise disjoint sets.
   The definitions of $R>\sqrt{2}|O|$ and $\Sigma(r)$ imply that for any $\s\in\Sigma(r)$
   \begin{align} \label{eq:diamO_leq_r}
     |O_\s|=r_\s|O|\leq \frac 1{\sqrt{2}} Rr_\s\leq \frac r{\sqrt{2}}.
   \end{align}
    Now fix $r>0$ and let $\s,\s'\in\Sigma(r)$ such that $O_\s(r)\cap O_{\s'}(r)\neq \emptyset$. Then, 
    by \eqref{eq:diamO_leq_r},
    \begin{align*}
       O_{\s'}\subset O_\s(2r+|O_{\s'}|)\subset B\left(f_\s(x), (2+\sqrt{2})r\right)\subset B\left(f_\s(x), 4r\right),
    \end{align*}
    where $x$ is an arbitrary point in $O$, and thus $f_\s(x)\in O_\s$. Here $B(y,s)$ denotes the closed ball with center $y$ and radius $s$.
    Recalling that the volume of a ball of radius $4r$ is $\kappa_d (4r)^d$ (where $\kappa_d$ is the volume of the unit ball in $\R^d$) and that the volume of each of the (pairwise disjoint) sets $O_{\sigma'}$, $\s'\in\Sigma(r)$ is bounded from below by
    \begin{align*}
      C_d(O_{\s'})=r_{\sigma'}^d C_d(O)> C_d(O) R^{-d} r_{\min}^d r^d,
    \end{align*}
    we conclude that not too many of the sets $O_{\s'}$ can be contained in the ball $B\left(f_\s(x), 4r\right)$. Hence we obtain that $\P$-almost surely
    \begin{align*}
      \sharp\left\{\s'\in\Sigma(r): O_\sigma(r)\cap O_{\sigma'}(r)\neq\emptyset\right\}
      &\leq \sharp\left\{\s'\in\Sigma(r): O_{\s'}\subseteq B(f_\s(x), 4r)\right\}\\
      &\leq \frac {\kappa_d 4^d R^d}{r_{\min}^d C_d(O)}=:\Gamma,
    \end{align*}
    where the constant $\Gamma$ is independent of $r>0$ and $\s$.
\end{proof}

Observe that, by UOSC, a.s.\ $F\subset\overline{O}$, which implies
   \begin{align}\label{eq:Fr_subset_Or}
      F(r)\subset O(r)\quad  \P\text{-a.s. for any } r>0
   \end{align}
   and hence, $F_\sigma(r)\subset O_\sigma(r)$ for any $\sigma\in\Sigma_*$. Therefore, the assertion of Lemma~\ref{lem:Gamma-bound} does equally hold with the sets $O_\sigma(r)\cap O_{\sigma'}(r)$ replaced by $F_\sigma(r)\cap F_{\sigma'}(r)$.

   Now we formulate the announced conditions that imply (ii) and (iii) in Theorem~\ref{maintheorem}. They are almost sure bounds in contrast to the bounds on expectations in (ii) and (iii). In the deterministic case, a condition very similar to the one in (b) is known to be equivalent to the curvature bounds corresponding to (ii) and (iii), cf.~\cite[Thm.~4.7]{Wi11}.

\begin{props}\label{prop:easier-cond}
Let $k\in\{0,1,\ldots,d\}$ and let $F$ be a homogeneous random fractal or a random self-similar set satisfying USOSC and the regularity condition {\rm(i)} in Theorem~\ref{maintheorem}.
Suppose that one of the following equivalent conditions {\rm (a)--(d)} holds:

\noindent {\rm(a)} there is a constant $c>0$ such that $\P$-a.s.\ for almost all $r>0$ and all $\s\in\Si(r)$,
\begin{align*}
   C_k^{\var}\bigg(F(r),F_\s(r)\cap\bigcup_{\tau\in\Si(r)\setminus\{\s\}}F_{\tau}(r)\bigg)\leq c r^k;
  \end{align*}
{\rm(b)} there is $c'>0$ such that $\P$-a.s.\ for a.a.\ $r>0$ and all $\s,\s'\in\Si(r)$ with $\s\neq\s'$,
\begin{align*}
  C_k^{\var}(F(r), F_{\s}(r)\cap F_{\s'}(r))\leq c' r^k;
\end{align*}
{\rm(c)}
there is $c''>0$ such that $\P$-a.s.\ for a.a.\ $r>0$ and all $\s\in\Si(r)$,
\begin{align*}
  C_k^{\var}(F(r), F_{\s}(r))\leq c'' r^k;
\end{align*}
{\rm(d)} there is $c'''>0$ such that $\P$-a.s.\ for a.a.\ $r>0$ and all $\s,\s'\in\Si(r)$ with $\s\neq\s'$,
\begin{align*}
  C_k^{\var}(F(r), O_{\s}(r)\cap O_{\s'}(r))\leq c''' r^k.
\end{align*}
Then the conditions {\rm (ii)} and {\rm (iii)} in Theorem~\ref{maintheorem} are satisfied.
\end{props}
\begin{proof}
Again we provide a proof for the homogeneous model, the one for the recursive case being similar. First we show the equivalence of the four conditions {\rm (a)--(d)}. The implications {\rm (a)\,$\Rightarrow$\,(b)}, {\rm (c)\,$\Rightarrow$\,(b)} and {\rm (d)\,$\Rightarrow$\,(b)} are obvious from corresponding set inclusions.

{\rm (b)\,$\Rightarrow$\,(a)}: Suppose (b) holds. Then $\P$-a.s.\ for a.a.\ $r>0$ and any $\s\in\Si(r)$,
  \begin{align*}
     C_k^{\var}\bigg(\,& F(r), F_\s(r)\cap \bigcup_{\tau\in\Si(r)\setminus\{\s\}} F_{\tau}(r)\bigg)
     = C_k^{\var}\bigg(F(r), \bigcup_{\tau\in\Si(r)\setminus\{\s\}}F_\s(r)\cap F_{\tau}(r)\bigg)\\
     &\leq \sum_{\tau\in\Si(r)\setminus\{\s\}} C_k^{\var}\big(F(r),F_\s(r)\cap F_{\tau}(r)\big)\leq \Gamma\cdot c' r^k,
  \end{align*}
  since, by Lemma~\ref{lem:Gamma-bound}, the number of nonzero summands in the last sum is bounded by some constant $\Gamma$ independent of $r$ or $\s$. Moreover, by {\rm(b)}, each of these summands is bounded by $c' r^k$.
 Hence, (b) implies (a) (with constant $c=\Gamma c'$), and therefore both conditions are equivalent.

 {\rm (a)\,$\Rightarrow$\,(c)}: We have $\P$-a.s., for a.a.\ $r>0$ and $\s\in\Sigma(r)$,
 \begin{align*}
 C_k^{\var}&\left(F(r), F_\s(r)\right)\leq C_k^{\var}\left(F(r), F_\s(r)\cap U_{r,\s}\right)+C_k^{\var}\left(F(r), F_\s(r)\setminus U_{r,\s}\right),
  \end{align*}
 where $U_{r,\s}:=\bigcup_{\tau\in\Si(r)\setminus\{\s\}}F_{\tau}(r)$. By {\rm(a)}, the first summand on the right is bounded by $c r^k$. For the second summand we infer that
 \begin{align*}
 C_k^{\var}\left(F(r), F_\s(r)\setminus U_{r,\s}\right)\leq C_k^{\var}\left(F(r), \R^d\setminus U_{r,\s}\right)=C_k^{\var}\left(F_\s(r),\R^d\setminus U_{r,\s}\right),
  \end{align*}
 by the local definiteness, and therefore, by Lemma~\ref{lems:large distances},
 \begin{align*}
 C_k^{\var}\left(F(r), F_\s(r)\setminus U_{r,\s}\right)\leq C_k^{\var}\left(F_\s(r)\right)\leq c_k(R) r^k,
  \end{align*}
  since, by definition of $\Sigma(r)$, $r\geq Rr_\sigma>\sqrt{2}|O|r_\sigma\geq \sqrt{2}|F_\sigma|$. Hence (c) holds (with $c''=c+c_k(R)$), showing the equivalence of (a) and (c).

  {\rm (c)\,$\Rightarrow$\,(d)}: We have $\P$-a.s., for a.a.\ $r>0$ and all $\s,\s'\in\Sigma(r)$ with $\s\neq\s'$, (with $U_{r,\s}$ defined as above 
  and noting that $F(r)=F_\s(r)\cup U_{r,\s}$)
 \begin{align*}
 C_k^{\var}&\left(F(r), O_\s(r)\cap O_{\s'}(r)\right)=C_k^{\var}\left(F(r), O_\s(r)\cap O_{\s'}(r)\cap \left(F_\s(r)\cup U_{r,\s}\right)\right)\\
 &\leq C_k^{\var}\left(F(r), O_\s(r)\cap O_{\s'}(r)\cap F_\s(r)\right)+C_k^{\var}\left(F(r), O_\s(r)\cap O_{\s'}(r)\cap U_{r,\s}\right)\\
 &\leq C_k^{\var}\left(F(r), F_\s(r)\right)+C_k^{\var}\left(F(r), O_\s(r)\cap U_{r,\s}\right)
  \end{align*}
  Due to (c), the first summand is bounded by $c'' r^k$, while for the second summand we get
  \begin{align*}
     C_k^{\var}\left(F(r), O_\s(r)\cap U_{\s,r}\right)&\leq C_k^{\var}\left(F(r), \bigcup_{\tau\in\Si(r)\setminus\{\s\}, O_{\tau}(r)\cap O_\s(r)\neq\emptyset} F_\tau(r)\right)\\
     &\leq \sum_{\tau\in\Si(r)\setminus\{\s\}, O_{\tau}(r)\cap O_\s(r)\neq\emptyset}C_k^{\var}\left(F(r), F_\tau(r)\right).
  \end{align*}
  Now, again by (c), each summand is bounded by $c'' r^k$ and, by Lemma~\ref{lem:Gamma-bound}, the number of summands is bounded by some constant $\Gamma$. Hence (d) holds (with $c'''=(1+\Gamma)c''$).

 It suffices now to show that condition (a) implies (ii) and (iii) in Theorem~\ref{maintheorem}.
  Applying condition (a) to the expectation in condition (ii), we conclude that
  for any $r_0\in(0,R)$, this expectation is bounded from above by
\begin{align*}
\E\esup_{r_0\le r\le R}\limits\,\max_{\s\in\Sb(r)} c r^k\leq c R^k<\infty.
\end{align*}
  Similarly, the conditional expectation in condition (iii) is bounded from above by
  \begin{align*}
\E\bigg[\max_{\s\in\Sb(r)}\, r^{-k}\,c r^k\bigg|\sharp\Sb(r)\bigg]\leq c\Gamma
\end{align*}
for a.a.\ $r\in(0,R)$. Hence conditions (ii) and (iii) are satisfied as claimed.
\end{proof}

\begin{figure}
  \includegraphics[width=0.45\textwidth]{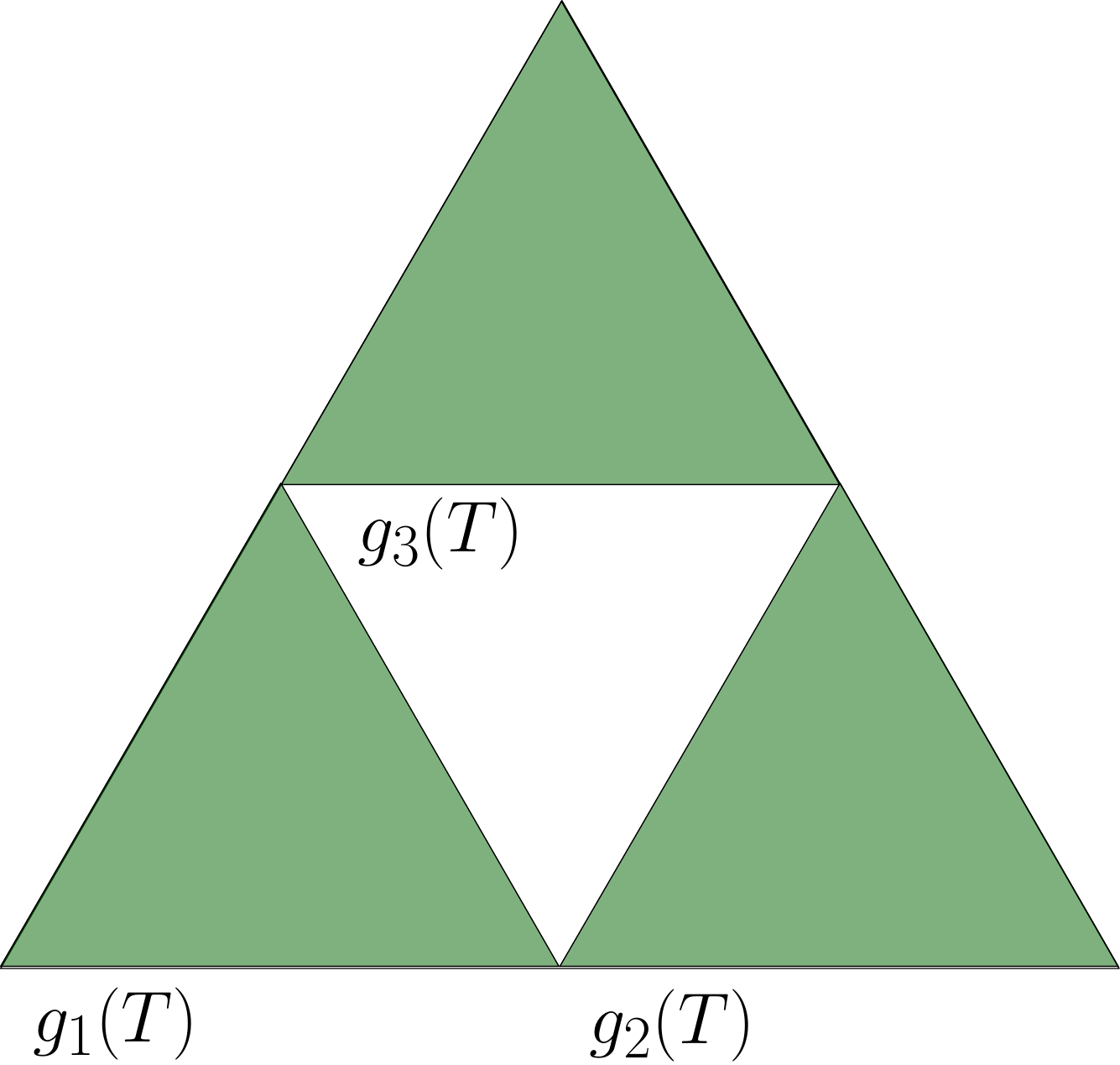}\hfill
  \includegraphics[width=0.45\textwidth]{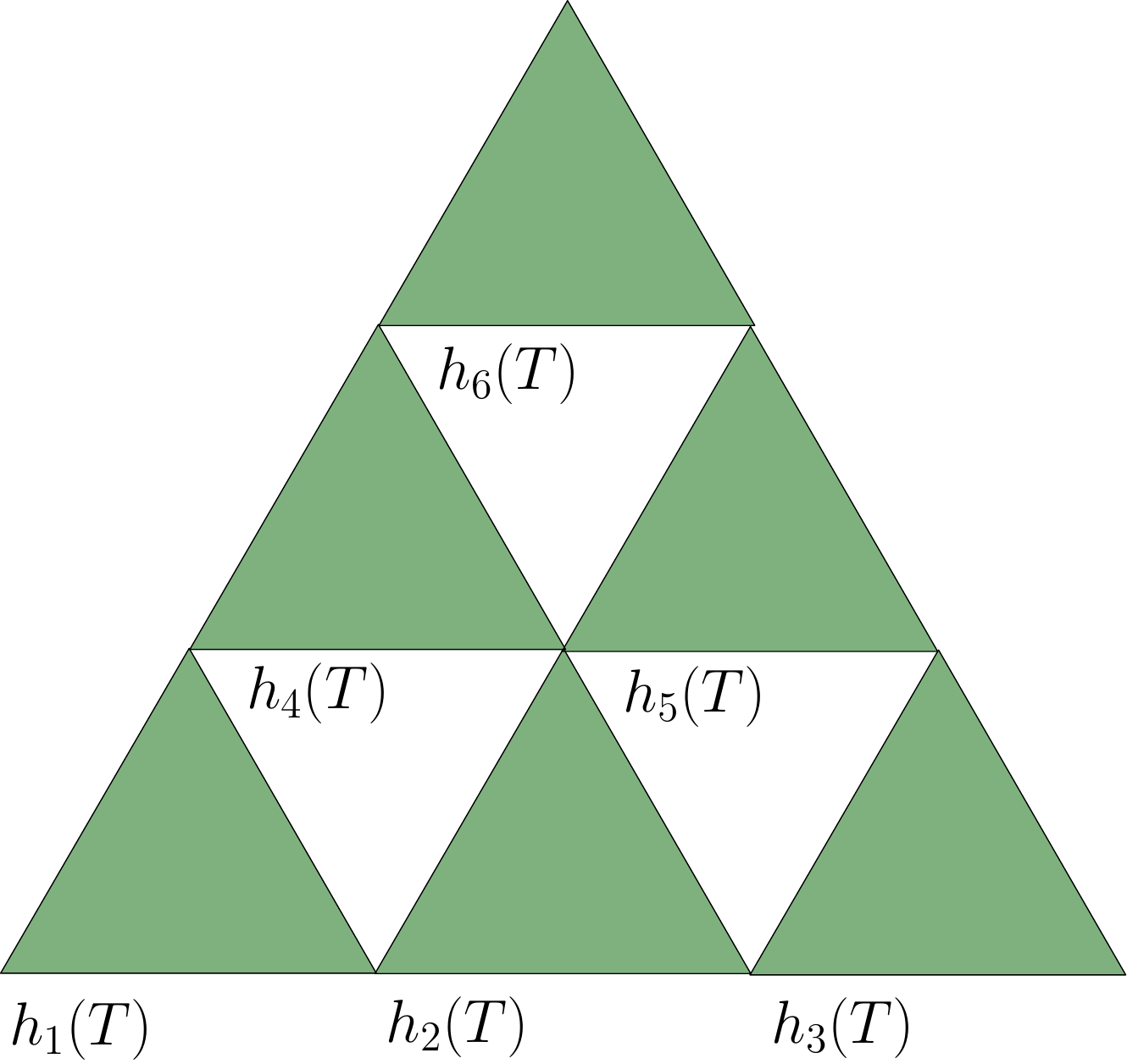}
  \caption{\label{fig:homSG} Illustration of the first construction step of the two IFS $G$ and $H$ in Example~\ref{ex:homSG} used to generate the homogeneous random Sierpi\'nski gasket.}
\end{figure}
\begin{ex} \label{ex:homSG}
  (Homogeneous random Sierpi\'nski gasket)
  Let $G=(g_1,g_2,g_3)$ be the IFS of the standard Sierpi\'nski gasket in $\R^2$, i.e.\ $g_i(x)=\frac 12 x+ t_i$, $x\in \R^2$, $i=1,2,3$, where $t_1=(0,0)$, $t_2=(\frac 12, 0)$ and $t_3=(\frac 14, \frac{\sqrt{3}}4)$, and let $H=(h_1,\ldots,h_6)$ be the IFS of the modified Sierpi\'nski gasket given by the six mappings $h_i(x)=\frac 13 x + s_i$, $x\in \R^2$, $i=1,\ldots,6$, where $s_1=(0,0)$, $s_2=(\frac 13,0)$, $s_3=(\frac 23,0)$, $s_4=(\frac 16,\frac{\sqrt{3}}6)$, $s_5=(\frac 12,\frac{\sqrt{3}}6)$ and $s_6=(\frac 13,\frac{\sqrt{3}}3)$, see also Figure~\ref{fig:homSG}.
  Then, for any $p\in[0,1]$, we consider the distribution $\P_0$ on the space $\Omega_0$ of the primary random IFS with
 \begin{align}\label{eq:ex_RIFS}
    \P_0(\{G\})=p, \qquad \P_0(\{H\})=1-p\, .
  \end{align}
 Let $F$ be the corresponding homogeneous random fractal generated as in Section \ref{sec:def_main_res}. Keep in mind that $\P_0$ and $F$ depend on the parameter $p$, although we suppress this dependence in the notation. The two deterministic self-similar sets generated by $G$ and $H$ are contained in this one-parameter family as marginal cases corresponding to $p=1$ and $p=0$, respectively.

First note that $F$ satisfies UOSC for the open set $O=\INt(T)$, where $T:=\conv\{(0,0),(1,0),(\frac 12, \frac{\sqrt{3}}2)\}$, since both $G$ and $H$ satisfy OSC for $O$.
Furthermore, $3<N\le 6$ with probability 1.
  We claim that the regularity condition (i) in Theorem~\ref{maintheorem} is satisfied. In fact, $\P$-a.s.\ all parallel sets of $F$ are polyconvex. (To see this, note that $\partial T\subset F\subset T$ and that for any $r\geq |T|=|F|$, $F(r)=T(r)$. Hence $F(r)$ is convex for these $r$. Now let $r>0$ be arbitrary and choose $n\in\N$ such that for all $\s\in\Si_n$, $|F_\s|\leq r$. (This is possible, since in each step sets are contracted at least by the factor $\frac 12$.) Then $F_\s(r)=T_\s(r)$ for all $\sigma\in\Sigma_n$ and, since $F(r)=\bigcup_{\s\in\Sigma_n} F_\s(r)$, we have found a representation of $F(r)$ by a finite number of convex sets. Hence $F(r)$ is polyconvex for each $r>0$.)

  In order to verify conditions (ii) and (iii), it suffices to check the assumptions of Proposition~\ref{prop:easier-cond}. First we need to specify the constant $R$. It is convenient to choose $R=\frac 32$. Since $|O|=1$, it clearly satisfies $R>\sqrt{2}|O|$. Let $r>0$ and recall the definition of the family $\Si(r)$. In the present example, the sets $F_\s$, $\s\in\Si(r)$, are all congruent copies of each other. In particular, they have the same diameter $r_\sigma$ and their parallel sets $F_\sigma(r)=T_\sigma(r)$ are convex, since $r\geq Rr_\sigma>\sqrt{2}|F_\sigma|$.

  Recall that $F(r)=\bigcup_{\tau\in\Si(r)} F_\tau(r)$ is a representation of $F(r)$ by convex sets. By Lemma~\ref{lem:Gamma-bound}, there is a constant $\Gamma$ such that locally within a fixed set $F_\s(r)$, $F(r)$ can be represented by at most $\Gamma$ of these sets. Let $\Si_\s(r):=\{\tau\in\Si(r):F_\tau(r)\cap F_\s(r)\neq\emptyset\}$. By the above considerations, $\sharp\Si_\s(r)\leq \Gamma$ and, since curvature measures are locally defined, we get
  \begin{align*}
     C_k^{\var}(F(r), F_{\s}(r)\cap F_{\s'}(r))
     &=C_k^{\var}\bigg(\bigcup_{\tau\in\Si_\s(r)} F_\tau(r), F_{\s}(r)\cap F_{\s'}(r)\bigg)\\
     &\leq C_k^{\var}\bigg(\bigcup_{\tau\in\Si_\s(r)} F_\tau(r)\bigg)\\
     &\leq 2^\Gamma \max_{\tau\in \Si_\s(r)} C_k(F_\tau(r))\\
     &\leq 2^\Gamma \max_{\tau\in \Si_\s(r)} r_\tau^k C_k(F^{(|\tau|)}(r^{-1}_\tau r))\\
     &\leq 2^\Gamma R^{-k} r^k C_k(B(0,1+6R))=:c r^k,
  \end{align*}
where in the third step we used \cite[Lemma~3.1.4]{Wi08} and in the last line the monotonicity of total curvatures for convex sets together with the fact that $|F(r_\tau^{-1}r)|\le|F|+2r_\tau^{-1}\le 1+2Rr_{\min}^{-1}\le 1+6R$.
Hence, by Proposition~\ref{prop:easier-cond}, conditions (ii) and (iii) of Theorem~\ref{maintheorem} are satisfied and we can apply this theorem.

Observe that the scaling exponent $D$ is given by the equation
\begin{align} \label{eq:D-ex}
  3p \cdot 2^{-D}+6(1-p) \cdot 3^{-D}=1,
\end{align}
and that the distribution of the logarithmic contraction ratios is
\begin{align*}
  \mu=\frac{3p}{2^D}\1_{(\cdot)}(\ln(2))+\frac{6(1-p)}{3^D}\1_{(\cdot)}(\ln(3)).
\end{align*}
Hence $\mu$ is non-lattice for $p\in(0,1)$ (and lattice for $p\in\{0,1\}$). Its mean value is given by
\begin{align} \label{eq:eta-ex}
  \eta=\frac{3p}{2^D}\ln(2)+\frac{6(1-p)}{3^D}\ln(3).
\end{align}
In order to compute the limits of the expected rescaled total curvatures, as provided by Theorem~\ref{maintheorem}, it remains, to determine
the functions $R_{k,L}$ for $k=0,1,2$.

For the computations it is convenient to set $L:=\frac{\sqrt{3}}6$ (which is the inradius of an equilateral triangle with sidelength 1).  We split the interval $(0,L)$ into three pieces. For
$r\in[L/2,L)$, all indicators in $R_{k,L}$ are zero almost surely, since $r_i\leq \frac 12$ a.s.\ Hence $R_{k,L}(r)=\E C_k(F(r))=C_k(T(r))$ in this case, since $r$ is large enough such that $F(r)$ has no holes. For $r\in[\frac L3,\frac L2]$, there are two possible situations. If we condition on the event that in the first step of the construction the IFS $G$ is chosen, then there will be a hole in $F(r)$. 
Moreover, the indicators in $R_{k,L}$ are 1, $N=3$, $r_i=\frac 12$ and the sets $F_i(r)=(g_iT)(r)$ are convex.
Since $F(r)=\bigcup_{i=1}^3 F_i(r)$ in this case, the inclusion-exclusion principle implies
\begin{align*}
  \E\big[\xi_k^L(r) \big| (f_1,\ldots,f_N)=G \big]&=\E \bigg[C_k(F(r))-\sum_{i=1}^3 C_k(F_i(r))\bigg|(f_1,\ldots,f_N)=G\bigg]\\
&= - 3 C_k((g_1T)(r)\cap (g_2T)(r)).
\end{align*}
Otherwise, i.e. if $H$ is chosen, we have $\E\big[\xi_k^L(r) \big|(f_1,\ldots,f_N)= H \big]=C_k(T(r))$.
This yields
\begin{align*}
R_{k,L}(r)=(1-p) C_k(T(r)) -p\cdot 3 C_k((g_1T)(r)\cap (g_2T)(r)), \qquad r\in[\frac L3,\frac L2].
\end{align*}
Finally, for $r\in(0,L/3)$, all indicators in $R_{k,L}$ are 1 a.s.\ and recalling that $F=\bigcup_{i=1}^N F_i$, the inclusion-exclusion principle implies that
\begin{align*}
  R_{k,L}(r)=\E \sum_{I\subset [N],\sharp I\geq 2} (-1)^{\sharp I-1} C_k\bigg(\bigcap_{i\in I} F_i(r)\bigg), \qquad r\in(0,\frac L3),
\end{align*}
where the summation is over all subsets $I$ of $[N]:=\{1,\ldots,N\}$ with at least two elements. Conditioning on either $G$ or $H$ being chosen in the first step of the construction and taking into account the symmetries of the resulting intersections, it is easily seen that
\begin{align*}
  -R_{k,L}(r)=p&\cdot 3\ C_k((g_1T)(r)\cap (g_2T)(r)))+(1-p)\cdot\\
  & \left(9 C_k((h_1T)(r)\cap (h_2T)(r))-C_k((h_2T)(r)\cap (h_4T)(r)\cap (h_5T)(r))\right).
\end{align*}
for any $r\in(0,L/3)$. Here we have used that under the condition $(f_1,\ldots,f_N)=G$, $F_1(r)\cap F_2(r)=(g_1T)(r)\cap(g_2T)(r)$ holds for any $r\in(0,L/3)$, and analogous relations under the condition $(f_1,\ldots,f_N)=H$.

Note that all the sets occurring above in the expressions for $R_{k,L}$ are nonempty and convex.
For $k=0$, this gives
\begin{align*}
  R_{0,L}(r)&=\begin{cases}
    5p-8, & r\in(0,L/3),\\
 1-4p,& r\in[L/3,L/2),\\
1,& r\in[L/2,L).
  \end{cases}
\end{align*}
Hence, for $p\in(0,1)$, Theorem~\ref{maintheorem} (I) yields
\begin{align*}
  C_{0,F}^{\fr}=\lim_{\ep\searrow 0} \ep^D \E C_0(F(\ep))=\frac{L^D}{D\cdot\eta}\left(1-(1-p) 3^{2-D} - p 2^{2-D}\right),
\end{align*}
where $L=\frac{\sqrt{3}}6$, $D$ is given by \eqref{eq:D-ex} and $\eta$ by \eqref{eq:eta-ex}. For $p\in\{0,1\}$, part (III) of Theorem~\ref{maintheorem} yields the same value for $\overline{C}_{0,F}^{\fr}$ (while the limit $ C_{0,F}^{\fr}$ does not exist in these two cases). For $p=1$, this recovers the value $-\frac{L^D}{3 \ln 3}$ obtained in \cite[Ex.~2.4.1]{Wi08}.

For $k=1$, recall that $C_1$ is half the boundary length. Therefore, we have
for any $r>0$, $C_1(T(r))=\frac 32 + \pi r$,
\begin{align*}
  C_1((g_1T)(r)\cap (g_2T)(r))=C_1((h_1T)(r)\cap (h_2T)(r))=(\frac 23 \pi+\sqrt{3})r
\end{align*}
and $C_k((h_2T)(r)\cap (h_4T)(r)\cap (h_5T)(r))=C_1(B(0,r))=\pi r$,
which yields
\begin{align*}
  R_{1,L}(r)&=\begin{cases}
    \left(3p(\pi+2\sqrt{3})-(9\sqrt{3}+5\pi)\right)\cdot r=:c_p\cdot r, & r\in(0,L/3),\\
 \frac 32(1-p)+ \underbrace{(\pi-3p(\pi+\sqrt{3})}_{=:\tilde c_p}\cdot r,
 & r\in[L/3,L/2),\\
\frac 32 +\pi\cdot r,& r\in[L/2,L).
  \end{cases}
\end{align*}
Plugging this into the formula (I) in Theorem~\ref{maintheorem}, we obtain
\begin{align*}
  C_{1,F}^{\fr}&=\lim_{\ep\searrow 0} \ep^{D-1} \E C_1(F(\ep))=\frac 1\eta \int_0^L r^{D-2} R_{1,L}(r) dr\\
  &=\frac{L^D}{D\eta}\left(3^{-D}(c_p-\tilde{c}_p) + 2^{-D}(\tilde c_p-\pi)+\pi\right)\\& \quad +\frac{3 L^{D-1}}{2(D-1)\eta}\left(1- (1-p)3^{1-D} -p2^{1-D}\right).
\end{align*}
 For $k=2$, we will derive below after Example~\ref{ex:recursiveSG} that
 $C_{2,F}^{\fr}=\frac 2{2-D} C_{1,F}^{\fr}$.
  \end{ex}

\begin{ex} \label{ex:recursiveSG}
  For $p\in(0,1)$, let $K=K_p$ be the random self-similar set generated by the same IFS-distribution $\P_0=\P_{0,p}$ (given by \eqref{eq:ex_RIFS}) as that for the homogeneous random fractal $F=F_p$ in Example~\ref{ex:homSG}. (Note that for $p\in\{0,1\}$ we get the same deterministic sets as above in Example~\ref{ex:homSG}.) By Remark~\ref{rem:det-and-rand-rec-case}, the mean fractal curvatures $C_{k,K}^{\fr}$ of $K$ are given by the same formulas as those of $F$. One just has to replace in the integrand the function $R_{k,L}$ for $F$ by the corresponding one for $K$. It is not difficult to see that 
  for all $k$ the functions $R_{k,L}$ for $K$ and $F$ coincide. Indeed, going through all the considerations in Example~\ref{ex:homSG}, it is clear that they apply equally to $K$. Hence, by \cite[Theorem~2.3.8]{Za11}, we get the same values for the mean fractal curvatures of $K$ as for $F$, i.e.\
  \begin{align*}
    C_{k,K}^{\fr}=C_{k,F}^{\fr}
  \end{align*}
  for $k=0,1,2$ and any $p\in(0,1)$. Moreover, by the same theorem, the almost sure limits $C_k(K):=\elim_{\ep\searrow 0} \ep^{D-k} C_k(K(\ep))$ exist and coincide with $C_{k,K}^{\fr}$, for any $p\in(0,1)$.
  \end{ex}

  Recall from \cite[Theorem 2.4]{RW13} that, for any bounded set $A\subset\R^d$ and any constants $s\in[0,d]$ and $M\in\R$, the limit $\lim_{r\searrow 0} \frac{C_d(A(r))}{r^{d-s}}$ exists and equals $M$ if and only if the limit $\lim_{r\searrow 0} \frac{ 2 C_{d-1}(A(r))}{(d-s)r^{d-1-s}}$ exists and equals $M$. Therefore, for any random self-similar set $K$ satisfying the assumptions of \cite[Theorem~2.3.8]{Za11}, the (almost sure) existence of the limits $C_k(K)$ for $k=d, d-1$ implies that almost surely
   \begin{align*}
     C_{d-1}(K)=\frac{d-D}2 \lim_{r\searrow 0} \frac{ 2 C_{d-1}(K(r))}{(d-D)r^{d-1-D}}=\frac{d-D}2 \lim_{r\searrow 0} \frac{C_d(K(r))}{r^{d-D}}=\frac{d-D}2 C_d(K)
   \end{align*}
   and thus also in the mean $
    C_{d-1,K}^{\fr}=\frac{d-D}2 C_{d,K}^{\fr}.
  $ 
  For the self-similar random set $K=K_p$ in Example~\ref{ex:recursiveSG} we obtain
     \begin{align*}
    C_{1,K}^{\fr}=\frac{2-D}2 C_{2,K}^{\fr},
  \end{align*}
   for any $p\in(0,1)$. Note that the observed coincidence of the mean fractal curvatures of $K$ with those of the homogeneous random fractal $F$ in Example~\ref{ex:homSG} implies now the same relation for $F$, i.e., we get
   similarly
   \begin{align*}
    C_{1,F}^{\fr}=\frac{2-D}2 C_{2,F}^{\fr},
  \end{align*}
  as claimed above.
  Note that in contrast the limit $\lim_{\ep\searrow 0} \ep^{D-k} C_k(F(\ep))$ (the corresponding $D$-dimensional Minkowski content of $F$) vanishes almost surely. Indeed, the a.s.\ Minkowski (and Hausdorff) dimension $D_H$ of $F$ is known to be given by the equation $\E\ln\bigg(\sum_{i=1}^N r_i^{D_H}\bigg)=0$ and thus strictly smaller than $D$, see e.g.~\cite{Za11}. It is an interesting open question whether the relation $C_{d-1,F}^{\fr}=\frac{d-D}2 C_{d,F}^{\fr}$ holds for any homogeneous random fractal $F$ with $D<d$.



\section{Appendix}

Given a nonempty compact set $K\subset\R^d$, we denote by
$$d_K:\, x\mapsto d(x,K):=\inf_{a\in K}|x-a|,\quad x\in\R^d,$$
the distance function to $K$. Note that $d_K$ is a $1$-Lipschitz function.
Denote by
$$\Sigma_K(x):=\{a\in K:\, |x-a|=d_K(x)\}$$
the set of all closest points of $K$ from $x$. By \cite[Lemma~4.2]{Fu85}, the Clarke subgradient $\partial d_K(x)$ of $d_K$ at $x$ equals
$$\partial d_K(x)=\conv(x-\Sigma_K(x)).$$
Recall that $x$ is called a \emph{regular point} of $d_K$ if $0\not\in\partial d_K(x)$. The above identity thus implies
\begin{equation} \label{reg_point}
x\not\in K\text{ is a regular point of }d_K\text{ if and only if }x\not\in\conv\Sigma_K(x).
\end{equation}
Moreover, Fu \cite{Fu85} showed that if $r>0$ is a \emph{regular value} of $d_K$ (i.e., all points $x$ with $d_K(x)=r$ are regular), then $\reach\widetilde{K(r)}>0$.

Given a nonempty set $A\subset\R^d$, we will use the notation $A^o:=\{y\in\R^d:\, y\cdot a\leq 0 \text{ for all } a\in A\}$ for the \emph{polar cone} of $A$. Note that the polar cone of the polar cone agrees with the generated convex cone:
$$A^{oo}=\left\{\sum_{i=1}^nt_ia_i:\, t_i\geq 0,\, a_i\in A,\, i=1,\dots,n,\, n\in\N\right\}.$$
If $\reach(\widetilde{K(r)},x)>0$, then the tangent and normal cones to $\widetilde{K(r)}$ at $x$ fulfill
\begin{align}
\Tan(\widetilde{K(r)},x)&=(\Sigma_K(x)-x)^o,  \label{Tan}\\
\Nor(\widetilde{K(r)},x)&=(\Sigma_K(x)-x)^{oo}.  \label{Nor}
\end{align}
Indeed, it is easy to see that if $a\in\Sigma_K(x)$, then $\INt B(a,|x-a|)\cap\widetilde{K(r)}=\emptyset$ and hence $a-x\in\Nor(\widetilde{K(r)},x)$, by \cite[Lemma~4.5]{RZ19}. Moreover, if $u\in \INt (\Sigma_K(x)-x)^o$, then the distance to $K$ increases (locally) in direction $u$ from $x$ implying that there is some $\ep>0$, such that the whole segment $[x,x+\ep u]$ is contained in $\widetilde{K(r)}$. Hence, $u\in\Tan(\widetilde{K(r)},x)$.

The above observations imply the characterization \eqref{rp2} of regular pairs $\Reg$ (see \eqref{regular pairs}): If $r>0$ is a regular value of $d_K$ then $\reach\widetilde{K(r)}>0$ and $\Nor(\widetilde{K(r)},x)\cap\rho(\Nor(\widetilde{K(r)},x))$ is trivial by \eqref{Nor}, hence $(r,K)\in\Reg$. For the reverse inclusion, if $(r,K)\in\Reg$ for some $r>0$, then $\reach\widetilde{K(r)}>0$ and if $x$ is a point with $d_K(x)=r$ then $\Nor(\widetilde{K(r)},x)\cap\rho(\Nor(\widetilde{K(r)},x))$ is trivial by the definition of $\Reg$ and, hence, due to \eqref{Nor}, $(\Sigma_K(x)-x)^{oo}$ contains no line through the origin, which implies by \eqref{reg_point} that $x\not\in\conv\Sigma_K(x)$, hence, $x$ is a regular point of $d_K$. Consequently, $r$ is a regular value of $d_K$.

In the rest of this section we provide a proof of Lemma~\ref{lem: contcurvmeas}. Due to \eqref{rp2},
it can be reformulated as follows.
\begin{thms}   \label{P_cont}
Let $K\subset\R^d$ be nonempty and compact and let $r_0>0$ be a regular value of $d_K$. Then, for any $k=0,1,\dots,d$, the weak convergence
$$\lim_{r\to r_0} C_k(K(r),\cdot)=C_k(K(r_0),\cdot)$$
takes place.
\end{thms}
We start with several auxiliary results.
Given $A\subset\R^d$ nonempty and $r>0$, we denote
$$A_{<r}:=\{y\in\R^d:\, d_A(y)< r\}.$$

\begin{lems}
Let $r>0$, $x\in\R^d$ and a compact set $\Sigma\subset \partial B(x,r)$ be given. Then
$$\INt B(s,|s-x|)\subset\Sigma_{<r},\quad s\in\conv\Sigma.$$
\end{lems}

\begin{proof}
Take any points $s\in\conv\Sigma$ and $y\in\INt B(s,|s-x|)$ and assume, for the contrary, that $d(y,\Sigma)\geq r$. Let $H$ be the hyperplane of symmetry of $y$ and $x$ and $H_+$ the closed half space with boundary $H$ and containing $x$. Then
$$\Sigma\subset\partial B(x,r)\cap\widetilde{B(y,r)}\subset H_+.$$
On the other hand, $H$ separates $x$ and $s$, hence, $s\not\in H_+$, which contradicts the assumption $s\in\conv\Sigma$.
\end{proof}

\begin{lems}  \label{L1}
If $K\subset\R^d$ is nonempty and compact and $r>0$ then
$$\reach\widetilde{K(r)}\geq \inf\{J_K(x):\, x\in\partial K(r)\},$$
where
$$J_K: x\mapsto \dist(x,\conv\Sigma_K(x)),\quad x\in\R^d\setminus K.$$
\end{lems}

\begin{proof}
Denote $\eta:=\inf\{J_K(x):\, x\in\partial K(r)\}$. If $\eta=0$ there is nothing to prove. Thus, assume that $\eta>0$.
We will show that
\begin{equation} \label{dTan}
d_{\Tan(\widetilde{K(r)},x)}(y-x)\leq\frac{|y-x|^2}{2\eta},\quad x,y\in\widetilde{K(r)}.
\end{equation}
This will imply the assertion, see \cite[Proposition~4.14]{RZ19}.

If $x\in\INt\widetilde{K(r)}$ then $\Tan(\widetilde{K(r)},x)=\R^d$ and \eqref{dTan} obviously holds. Hence, assume that $x\in\partial\widetilde{K(r)}$, which implies that $d_K(x)=r$.

Denote $T:=(\Sigma_A(x)-x)^o$, $N:=T^o$, and recall that $T=\Tan(\widetilde{K(r)},x)$ and $N=\Nor(\widetilde{K(r)},x)$ (see \eqref{Tan}, \eqref{Nor}).
Denote further $u:=p_T(y-x)$ and $v:=p_N(y-x)$ ($p_T$, $p_N$ denote the orthogonal projection to $T$, $N$, respectively), and note that $u\cdot v=0$ since $T,N$ are dual convex cones. We will show that
$$|(y-x)-u|\leq \frac{|y-x|^2}{2\eta}$$
whenever $y\in\widetilde{K(r)}$, which will prove \eqref{dTan}. If $v=0$ then $y-x=u$ and we are done. If $v\neq 0$ denote $\bar{v}:=\frac{v}{|v|}$ and note then, using elementary planar geometry,
$$|(y-x)-u|=(y-x)\cdot\bar{v}.$$
Hence, we have to show that
\begin{equation} \label{E_fin}
(y-x)\cdot\bar{v}\leq \frac{|y-x|^2}{2\eta}.
\end{equation}
Let $s\in\conv\Sigma_K(x)$ be such that $s-x=t\bar{v}$ for some $t>0$. Applying Lemma~\ref{L1} we obtain
$$\INt B(x+\eta\bar{v},\eta)\subset\INt B(s,|s-x|)\subset K_{<r}=\R^d\setminus\widetilde{K(r)},$$
which means that
$$\eta^2\leq |y-x-\eta\bar{v}|^2=|y-x|^2+\eta^2-2\eta(y-x)\cdot\bar{v},$$
and this proves \eqref{E_fin}.
\end{proof}

\begin{lems} \label{semicont}
The function $x\mapsto J_K(x)$ is lower semicontinuous on $\R^d$, i.e.,
$$\liminf_{y\to x}J_K(y)\geq J_K(x).$$
\end{lems}

\begin{proof}
First we show that the set-valued function $x\mapsto\Sigma_K(x)$ is upper semicontinuous (w.r.t.\ Hausdorff metric), i.e., that $\limsup_{y\to x}\Sigma_K(y)\subset\Sigma_K(x)$. (This was proved in \cite[Lemma~5.1]{RSM09}. We repeat the short argument here for the convenience of the reader.)
Let $x_n\to x$, $a_n\in\Sigma_K(x_n)$, $a_n\to a$. We will show that $a\in\Sigma_K(x)$. If not, there would be another point $b\in K$ with $|b-x| < |a-x|$. Let $n$ be sufficiently large that
$$\max\{|a_n-a|, |x_n-x|\} < \ep:=\frac 13 (|a-x|-|b-x|).$$
Then, by the triangle inequality,
\begin{align*}
|x_n-b| &\leq |x_n-x| + |b-x| < \ep + |a-x|-3\ep\\
&< |a-x|-|a-a_n|-\ep \leq |a_n-x|-\ep\\
&< |a_n-x|-|x-x_n| \leq |a_n-x_n|,
\end{align*}
which means that $a_n$ is not the closest point of $K$ to $x_n$, a contradiction.

Since the convex hull is a continuous operation on compact sets, also
$$\limsup_{y\to x}\conv(\Sigma_K(y)-y)\subset\conv(\Sigma_K(x)-x).$$
This means that for any $\ep>0$ there exists $\delta>0$ such that

$$\conv(\Sigma_K(y)-y)\subset(\conv(\Sigma_K(x)-x))(\ep),\quad |y-x|<\delta,$$
which clearly implies that $J_K(y)\geq J_K(x)-\ep$ whenever $|y-x|<\delta$.
\end{proof}

\begin{proof}[Proof of Theorem~\ref{P_cont}]
For $k=d$ see Remark \ref{rems: large k}.

Assume now that $k\leq d-1$. Since the set of regular values of $d_K$ is open, there exists an $\ep>0$ such that for any $r\in(r_0-\ep,r_0+\ep)$, $r$ is a regular value of $d_K$, $\reach\widetilde{K(r)}>0$ and \eqref{reflect} holds. Thus, in order to prove the statement, it is enough to show that
$$\lim_{r\to r_0} C_k(\widetilde{K(r)},\cdot)=C_k(\widetilde{K(r_0)},\cdot).$$
Since clearly $K(r)\to K(r_0)$ in the Hausdorff distance as $r\to r_0$, it will be enough to show that
\begin{equation} \label{liminf}
\liminf_{r\to r_0}\reach\widetilde{K(r)}>0,
\end{equation}
and apply \cite[\S5.9]{Fe59}. The function $J_K$ is positive on $\partial K(r_0)$ and, using Lemma~\ref{semicont}, we obtain that for any $x\in\partial K(r_0)$ there exists a $\delta(x)>0$ such that $J_K(y)>J_K(x)/2$ whenever $y\in U(x,\delta_x)$. By the compactness of $\partial K(r_0)$, we easily find an $\eta>0$ and an open set $U\subset\partial K(r_0)$ such that $J_K>\eta$ on $U$. Since $\partial K(r)\subset U$ for $r$ sufficiently close to $r_0$, we conclude that
$$\liminf_{r\to r_0}\,\inf\{J_K(x):\, x\in\partial K(r)\}\geq\eta,$$
and the proof is completed by applying Lemma~\ref{L1}.
\end{proof}
At the end we show the measurability property used in the proof of Lemma \ref{lem:measurability}.
By a \emph{random signed measure} we understand a mapping $\mu$ from a probability space into the space of locally finite signed Borel measures such that $\mu(B)$ is a random variable for any bounded Borel set $B$. Recall that the space $\K$ of nonempty compact sets was provided with the Borel $\s$-algebra determined by the Hausdorff distance.
\begin{lems}  \label{L_prod_meas}
Let $\mu=\mu^\omega$ be a random signed measure. Then the mapping
$$\Phi: (\omega,K)\mapsto\mu^\omega(K),\quad (\omega,K)\in\Omega\times\K$$
is jointly measurable.
\end{lems}

\begin{proof}
Note that the mapping
$$\Psi: (\omega,f)\mapsto\int f\, d\mu^\omega,\quad (\omega,f)\in\Omega\times\C_c$$ is jointly measurable ($\C_c$ is the space of continuous functions with compact support with supremum metric). This follows from the fact that $\omega\mapsto\int f\, d\mu^\omega$ is measurable and $f\mapsto\int f\, d\mu^\omega$ is continuous, hence $\Psi$ is a Caratheodory function, which is always jointly measurable, see \cite[Lemma~4.51]{AB06}.

Further, given $K\in\K$ and $n\in\N$, consider the function
$$f_{K,n}: x\mapsto (1-nd_K(x))^+,\quad x\in\R^d.$$
Since for any $K,K'\in\K$ we have $|f_{K,n}(x)-f_{K',n}(x)|\leq nd_H(K,K')$, the mapping $\varphi_n: K\mapsto f_{K,n}$ is continuous from $\K$ to $\C_c$. Thus, the mapping $\tilde{\varphi}_n: (\omega,K)\mapsto(\omega,f_{K,n})$, as well as the composition $\Phi_n:=\Psi\circ\tilde{\varphi}_n$, are measurable on the product space $\Omega\times\K$. Since $\Phi_n(\omega,K)=\int f_{K,n}\, d\mu^\omega\to \mu^\omega(K)=\Phi(\omega, K)$, $n\to\infty$, for any $(\omega, K)$, the limit function $\Phi$ is jointly measurable as well.
\end{proof}


\end{document}